\documentclass[10pt]{amsart}

\usepackage{amsfonts}
\usepackage{amssymb}
\usepackage{amsmath, amsthm, latexsym}
\usepackage[all]{xy}
\usepackage{hyperref}
\usepackage{graphicx, color}
\usepackage[margin=1.5in]{geometry}
\usepackage{tikz, tikz-cd}
\usetikzlibrary{calc}

\def \C{\mathbb{C}}
\def \Z{\mathbb{Z}}
\def \R{\mathbb{R}}

\def \E{\mathcal{E}}

\def \G{\mathbb{G}}
\def \v{{\bf v}}

\def \k{{\bf k}}

\def \B{\mathfrak{B}}
\def \L{\mathcal{L}}
\def \P{\mathcal{P}}

\def \Spec{\operatorname{Spec}}

\def \SL{\operatorname{SL}}
\def \GL{\operatorname{GL}}
\def \End{\operatorname{End}}

\def \Lie{\operatorname{Lie}}

\def \Ad{\operatorname{Ad}}
\def \diag{\operatorname{diag}}

\def \rank{\operatorname{rank}}

\theoremstyle{plain}
\newtheorem{theorem}{Theorem}[section]
\newtheorem{lemma}[theorem]{Lemma}
\newtheorem{proposition}[theorem]{Proposition}
\newtheorem{corollary}[theorem]{Corollary}

\newtheorem{THM}{Theorem}
\newtheorem{CORO}[THM]{Corollary}

\theoremstyle{definition}
\newtheorem{example}[theorem]{Example}
\newtheorem{definition}[theorem]{Definition}
\newtheorem{remark}[theorem]{Remark}

\newtheorem*{DEFIN}{Definition}
\newtheorem*{REM}{Remark}

\newtheorem*{QUESTION}{Question}

\begin{document}
\title{Toric principal bundles, piecewise linear maps and Tits buildings}

\author{Kiumars Kaveh}
\address{Department of Mathematics, University of Pittsburgh,
Pittsburgh, PA, USA.}
\email{kaveh@pitt.edu}

\author{Christopher Manon}
\address{Department of Mathematics, University of Kentucky, Lexington, KY, USA}
\email{Christopher.Manon@uky.edu}

\date{\today}
\thanks{The first author is partially supported by National Science Foundation Grants (DMS-1601303 and DMS-2101843) and a Simons Collaboration Grant (award number 714052). Also the second author is partially supported by National Science Foundation Grants (DMS-1500966 ans DMS-2101911) and a Simons Collaboration Grant (award number 587209).}
\subjclass[2010]{}
\begin{abstract}
We define the notion of a piecewise linear map from a fan $\Sigma$ to $\tilde{\B}(G)$, the cone over the Tits building of a linear algebraic group $G$. Let $X_\Sigma$ be a toric variety with fan $\Sigma$. We show that when $G$ is reductive the set of integral piecewise linear maps from $\Sigma$ to $\tilde{\B}(G)$ classifies the isomorphism classes of (framed) toric principal $G$-bundles on $X_\Sigma$. This in particular recovers Klyachko's classification of toric vector bundles, and gives new classification results for the orthogonal and symplectic toric principal bundles. 

\end{abstract}

\maketitle



\tableofcontents

\section*{Introduction}
In this paper we give a classification of torus equivariant principal bundles on toric varieties (or {\it toric} principal bundles for short). It extends Klyachko's well-known classification of torus equivariant vector bundles on toric varieties \cite{Klyachko}. Klyachko's classification itself is an extension of the classification of equivariant line bundles on toric varieties by integral piecewise linear functions. We show that, for a reductive algebraic group $G$, toric principal $G$-bundles are classified by the data of \emph{integral piecewise linear maps} to $\tilde{\B}(G)$, the \emph{cone} over (the underlying space of) the Tits building  of $G$. 

The first classification of toric vector bundles goes back to Kaneyama \cite{Kaneyama} and is in terms of certain cocycles. On the other hand, Klyachko's classification in \cite{Klyachko} is in terms of certain compatible filtrations on a finite dimensional  vector space. As for toric principal bundles, recently in the interesting papers \cite{Biswas, Biswas-Tannakian} the authors give a classification of toric principal bundles using certain data of homomorphisms and cocycles as well as certain filtered algebras. It is not immediately clear (at least to us) that their data defines a piecewise linear map to the cone over the building $\tilde{\B}(G)$ (in the sense of this paper). The classification in \cite{Biswas} seems to be in the spirit of Kaneyama. The classification in the present paper is in the spirit of Klyachko. We also mention \cite{Payne-cover} where the author considers a family of Klyachko filtrations parametrized by vectors in the fan. This is basically the same as our piecewise linear map in the case of toric vector bundles.



We start with a brief conceptual explanation for the appearance of buildings in the classification of toric principal bundles: 
as a corollary of the Luna slice theorem, one shows that torus equivariant principal $G$-bundles on a toric variety are trivial over toric affine charts and in such a chart the torus acts via a homomorphism from the torus to $G$ (see Definition \ref{def-local-equiv-trivial} and Theorem \ref{th-equiv-trivial-affine}). The image of a torus under a homomorphism necessarily lands in a maximal torus, and the arrangement of maximal tori in $G$ is encoded in its Tits building.
Our classification of toric principal bundles crucially relies on a realization of the Tits building of $G$ as the set of one-parameter subgroups of $G$ modulo certain equivalence relation (Section \ref{subsec-1-para-subgp}).



Throughout $\k$ denotes the ground field which we take to be algebraically closed. Let $\Sigma$ be a fan in $\R^n$ with $X_\Sigma$ its associated toric variety. Recall that a toric variety is a normal variety equipped with an action of algebraic torus $T$ such that $T$ has an open orbit isomorphic to $T$ itself.  A toric line bundle $\L$ on $X_\Sigma$ is a line bundle with a $T$-linearization. 
It is well-known that toric line bundles on $X_\Sigma$ are in one-to-one correspondence with functions $\phi: |\Sigma| \to \R$ that are piecewise linear with respect to $\Sigma$ and are integral, i.e. map $|\Sigma| \cap \Z^n$ to $\Z$ (here $|\Sigma|$ denotes the support of $\Sigma$, the union of all cones in $\Sigma$). A \emph{toric vector bundle} $\E$ on a toric variety $X_\Sigma$ is a vector bundle equipped with a $T$-linearization. A \emph{toric principal $G$-bundle} on $X_\Sigma$, where $G$ is a linear algebraic group, is a principal $G$-bundle together with a $T$-action that commutes with the $G$-action. We often write the $T$-action on the left and the $G$-action on the right. The isomorphism classes of rank $r$ toric vector bundles are in natural bijection with the isomorphism classes of toric principal $\GL(r)$-bundles.


Recall that a {\it building} is an abstract simplicial complex together with a collection of distinguished subcomplexes, called {\it apartments}, satisfying certain axioms (Definition \ref{def-building}). 
To a linear algebraic group $G$, there corresponds a building  $\Delta(G)$ whose simplices correspond to parabolic subgroups of $G$, its maximal simplices (called \emph{chambers}) correspond to Borel subgroups, and its apartments correspond to maximal tori in $G$. We remark that, since every parabolic subgroup contains the solvable radical of $G$, the building $\Delta(G)$ and that of its semisimple quotient, i.e. $\Delta(G_{\textup{ss}})$, can naturally be identified. So as, far as the simplicial complex structure of the building is concerned, we can always assume that $G$ is semisimple.

Let $G$ be semisimple. The abstract simplicial complex $\Delta(G)$ has a natural geometric realization. That is, there is a topological space $\B(G)$ such that each simplex in $\Delta(G)$ can be identified with a subset of $\B(G)$ homeomorphic to a standard simplex 
and these simplices are glued along their common subsimplices (see Section \ref{sec-buildings}). It is constructed as follows. For each maximal torus $H \subset G$ let $\Lambda^\vee(H)$ be its cocharacter lattice and let $\Lambda^\vee_\R(H)= \Lambda^\vee(H) \otimes_\Z \R$. The apartment corresponding to $H$ is the triangulation of the unit sphere in 
$\Lambda^\vee_\R(H)$ obtained by intersecting it with the Weyl chambers and their faces. Two simplices, in different apartments, are glued together if the corresponding faces represent the same parabolic subgroup in $G$. The topological space $\B(G)$ is obtained by gluing the unit spheres in the $\Lambda^\vee_\R(H)$, for all maximal tori $H$, along their common simplices. Similarly, we construct the topological space $\tilde{\B}(G)$ by gluing the vector spaces $\Lambda^\vee_\R(H)$, along their common faces of Weyl chambers. We think of $\tilde{\B}(G)$ as the cone over the topological space $\B(G)$.

While in our notation, we distinguish between the building as an abstract simplicial complex, i.e. $\Delta(G)$, and as a topological space, i.e. $\B(G)$, by abuse of terminology we refer to both $\Delta(G)$ and $\B(G)$ as the Tits building of $G$. 

Now let $G$ be a linear algebraic group and let $G_{\textup{ss}} = G / R(G)$ be the semisimple quotient of $G$. The previous construction in the semisimple case works in this case as well and we can define $\tilde{\B}(G)$ (respectively $\B(G)$) to be the topological space obtained by gluing the vector spaces $\Lambda^\vee_\R(H)$ (respectively unit spheres in the $\Lambda^\vee_\R(H)$), for all maximal tori $H \subset G$, along their common faces of Weyl chambers (respectively intersections of common faces with the unit spheres). 
When $G$ is reductive, the topological space $\tilde{\B}(G)$ is the Cartesian product of $\tilde{\B}(G_{\textup{ss}})$ with the real vector space $\Lambda^\vee(Z) \otimes \R$, where $Z=Z(G)^\circ$ is the connected component of the identity in the center of $G$ (see Definitions \ref{def-B} and \ref{def-tilde-B}). 

By slight abuse of terminology, for a linear algebraic group $G$, we refer to $\tilde{\B}(G)$ as the \emph{cone} over the Tits building of $G$ (strictly speaking, it is the \emph{cone} over the Tits building only in the semisimple case). Also, for a maximal torus $H$, we refer to $\Lambda^\vee_\R(H)$ as the \textit{cone} over the apartment of $H$ and denote it by $\tilde{A}_H$. 

We will see in Section \ref{subsec-1-para-subgp} that the set of lattice points in $\tilde{\B}(G)$ can be identified with the set of one-parameter subgroups of $G$ modulo certain equivalence relation (Definition \ref{def-1-para-equiv} and Corollary  \ref{cor-tilde-B-one-para}).

The following is the main definition in the paper (Definition  \ref{def-plm-to-B(G)}).
\begin{DEFIN}
We say that a function $\Phi: |\Sigma| \to \tilde{\B}(G)$ is a {\it piecewise linear map with respect to $\Sigma$} if the following hold:
\begin{itemize}
\item[(1)] For each cone $\sigma \in \Sigma$, the image $\Phi(\sigma)$ lies in a cone over an apartment $\tilde{A}_\sigma = \Lambda_\R^\vee(H_\sigma)$ (which of course is not necessarily unique). Here $H_\sigma \subset G$ is the corresponding maximal torus. 
\item[(2)] For each $\sigma \in \Sigma$, the restriction 
$\Phi_\sigma := \Phi_{|\sigma}: \sigma \to \tilde{A}_\sigma$ is an $\R$-linear map.
\end{itemize}
We call $\Phi$ an {\it integral} piecewise linear map if for each $\sigma \in \Sigma$, $\Phi_\sigma$ sends $\sigma \cap \Z^n$ to $\Lambda^\vee(H_\sigma)$.   

Let $G$, $G'$ be reductive groups. A homomorphism of algebraic groups $\alpha: G \to G'$ induces a natural map $\hat{\alpha}: \tilde{\B}(G) \to \tilde{\B}(G')$ (see Definition \ref{def-B(G)-B(G')}). If $\Phi: |\Sigma| \to \tilde{\B}(G)$ is a piecewise linear map then $\alpha_*(\Phi) := \hat{\alpha} \circ \Phi: |\Sigma| \to \tilde{\B}(G')$ is also a piecewise linear map. 
\end{DEFIN}

Throughout the paper we use the term ``piecewise linear function'' for a function with values in $\R$ and the term ``piecewise linear map'' for a function with values in $\tilde{\B}(G)$ (see Example \ref{ex-tlb} for how a piecewise linear function is a special case of a piecewise linear map).

Let $X_\Sigma$ be a toric variety. For the rest of the paper, we fix a point $x_0$ in the open torus orbit in $X_\Sigma$. By a {\it framed} toric principal $G$-bundle we mean a toric principal $G$-bundle $\P$ together with a choice of a point $p_0 \in \P_{x_0}$. The choice of $p_0 \in \P_{x_0}$ is equivalent to fixing a $G$-isomorphism between $\P_{x_0}$ and $G$ (regarded as a $G$-variety by multiplication from right). 
Let $\alpha: G \to G'$ be a homomorphism of algebraic groups. Let $\P$ (respectively $\P'$) be a framed toric principal $G$-bundle (respectively a framed toric principal $G'$-bundle) on $X_\Sigma$. A \emph{morphism} from $\P$ to $\P'$ is a morphism of bundles that is equivariant with respect to $\alpha$ and commutes with the $T$-action, and moreover, sends the frame $p_0 \in \P_{x_0}$ to the frame $p'_0 \in \P'_{x_0}$ (see the second paragraph after Definition \ref{def-plm-to-B(G)}).

The main result of the paper is the following (see Theorem \ref{th-principal-G-bundle-plm}). 

\begin{THM}[Main theorem]   \label{th-intro-principal-bundle}
Let $G$ be a reductive algebraic group over $\k$.
There is a one-to-one correspondence between the isomorphism classes of framed toric principal $G$-bundles $\P$ over $X_\Sigma$ and the integral piecewise linear maps $\Phi: |\Sigma| \to \tilde{\B}(G)$. 
{Moreover, let $\alpha:G \to G'$ be a homomorphism of reductive algebraic groups. Let $\mathcal{P}$ (respectively $\mathcal{P}'$) be a framed toric principal $G$-bundle (respectively $G'$-bundle) with corresponding piecewise linear map $\Phi: |\Sigma| \to \tilde{\B}(G)$ (respectively $\Phi': |\Sigma| \to \tilde{\B}(G')$). Then there is a morphism of framed toric principal bundles $F: \mathcal{P} \to \mathcal{P}'$, that is equivariant with respect to $\alpha$, if and only if $\Phi' = \alpha_*(\Phi)$.}
\end{THM}

In fact, we prove a more general version of Theorem \ref{th-intro-principal-bundle} that does not require $G$ to be reductive. Let us say that a toric principal $G$-bundle $\mathcal{P}$ over a toric variety $X_\Sigma$ is \textit{locally equivariantly trivial} if the following holds: for any toric affine chart $X_\sigma$, $\sigma \in \Sigma$, the restriction $\mathcal{P}_{|X_\sigma}$ is $T$-equivariantly isomorphic (as a principal $G$-bundle) to the trivial bundle $X_\sigma \times G$ where $T$ acts on $X_\sigma \times G$ diagonally, acting on $X_\sigma$ in the usual way and on $G$ via a homomorphism $\phi_\sigma: T \to G$ (see Definition \ref{def-local-equiv-trivial}).
 
\begin{THM}[Main theorem, second version] \label{th-intro-principal-bundle-v2}
Let $G$ be a linear algebraic group over $\k$. There is a one-to-one correspondence between the isomorphism classes of framed toric principal $G$-bundles $\P$ over $X_\Sigma$ that are locally equivariantly trivial and the integral piecewise linear maps $\Phi: |\Sigma| \to \tilde{\B}(G)$. 
{Moreover, let $\alpha:G \to G'$ be a homomorphism of linear algebraic groups. Let $\mathcal{P}$ (respectively $\mathcal{P}'$) be a locally equivariantly trivial framed toric principal $G$-bundle (respectively $G'$-bundle) with corresponding piecewise linear map $\Phi: |\Sigma| \to \tilde{\B}(G)$ (respectively $\Phi': |\Sigma| \to \tilde{\B}(G')$). Then there is a morphism of framed toric principal bundles $F: \mathcal{P} \to \mathcal{P}'$, that is equivariant with respect to $\alpha$, if and only if $\Phi' = \alpha_*(\Phi)$.
}
\end{THM}


When $G$ is a reductive algebraic group the local equivariant triviality for toric principal $G$-bundles is proved in \cite[Theorem 4.1]{Biswas-Tannakian} which in turn relies on \cite[Prop. 8.5]{BR} proved using a Luna slice argument. Hence Theorem \ref{th-intro-principal-bundle} follows from Theorem \ref{th-intro-principal-bundle-v2}.

We point out that the local equivariant triviality of a toric \textit{vector bundle} over any field is easy to prove (Proposition \ref{prop-toric-vb-over-affine-equiv-trivial} and \cite[Proposition 2.1.1]{Klyachko}). 
We also remark that the local triviality of toric principal $G$-bundles is not an immediate corollary of the local triviality of toric vector bundles. Let $G \subset \GL(E)$ be a closed reductive subgroup. Then a toric principal $G$-bundle $\P$ on an affine toric variety $X_\sigma$ has an associated toric vector bundle $\E = \P \times_G E$ and $\P$ is a subbundle of the frame bundle of $\E$. The equivariant triviality of $\E$ implies the equivariant triviality of its frame bundle but it does not immediately imply the equivariant triviality of the subbundle $\P$.

When the base field is $\C$, in \cite{Biswas-Serre} the local equivariant triviality is proved  for any linear algebraic group $G$. Hence when $\k = \C$, Theorem \ref{th-intro-principal-bundle} holds for any linear algebraic group $G$ over $\C$ as well. We ask the following question:

\begin{QUESTION}
For an algebraically closed field $\k$ of positive characteristic, can one give an example of a non-reductive linear algebraic group $G$ and a toric principal $G$-bundle over an affine toric variety that is not equivariantly trivial?
\end{QUESTION}


When $G$ is the general linear group, Theorem \ref{th-intro-principal-bundle} recovers Klyachko's classification of toric vector bundles (see Section \ref{sec-prelim-toric-vb} and Example  \ref{sec-toric-vb-PLM}). 




Theorem  \ref{th-intro-principal-bundle} readily gives Klyachko type classifications for toric principal bundles for other classical groups such as the orthogonal and symplectic groups (see Example \ref{ex-building-orth-symp} and Example \ref{ex-orth-symp-bundle}). We only discuss the cases of symplectic and even orthogonal groups. The odd orthogonal group case can be dealt with in a similar fashion.  Let $\langle \cdot, \cdot \rangle$ be a symmetric or skew-symmetric bilinear form on a vector space $E \cong \k^{2r}$. Let $G$ be the group of linear isomorphisms of $E$ preserving $\langle \cdot, \cdot \rangle$. Thus $G \cong \operatorname{O}(2r)$ or $\operatorname{Sp}(2r)$. Consider a flag of subspaces:
$$F_\bullet = (\{0\} \subsetneqq F_1 \subsetneqq \cdots \subsetneqq F_{k} = E).$$ We say $F_\bullet$ is an {\it isotropic flag} if  
for each $j=1, \ldots, k$ we have $F_j^{\perp} = F_{k-j}$. An {\it integral labeled isotropic flag} $(F_\bullet, c_\bullet)$ is an isotropic flag $F_\bullet$ whose subspaces are labeled by a decreasing sequence of integers $c_\bullet = (c_1 > \cdots > c_{k})$ such that $c_j = -c_{k+1-j}$, for $j=1, \ldots, k$. Alternatively, we can think of an integral labeled isotropic flag as a $\Z$-filtration of isotropic subspaces in $E$ (see Example \ref{ex-building-orth-symp}).

\begin{CORO}[Toric principal bundles for symplectic and orthogonal groups]     \label{cor-intro-orth-symp-bundle}
Let $G = \operatorname{O}(2r)$ or $\operatorname{Sp}(2r)$.	
With notation as above, the isomorphism classes of toric principal $G$-bundles on $X_\Sigma$ are in one-to-one correspondence with the collections $\{ (F_{\rho, \bullet}, c_{\rho, \bullet}) \mid \rho \in \Sigma(1) \}$ of integral labeled isotropic flags that satisfy the following compatibility condition:  for each cone $\sigma \in \Sigma$, there exists a normal frame 
$L_\sigma = \{L_{\sigma, 1}, \ldots, L_{\sigma, 2r} \}$ and a linear map $\Phi_\sigma: (\sigma \cap N) \to \Z^r$ such that for each ray $\rho \in \sigma(1)$ the labeled isotropic flag associated to $(L_\sigma, \Phi_\sigma(\v_\rho))$ coincides with $(F_{\rho, \bullet}, c_{\rho, \bullet})$ (see Example \ref{ex-building-orth-symp} for the notion of normal frame).  

\end{CORO}

\textbf{A little application.}
The following illustrates an application of our building approach. Let $G$ be a reductive algebraic subgroup with $K \subset G$ a closed subgroup. We say that a toric principal $G$-bundle has an \emph{equivariant reduction of structure group to $K$} if there is a toric principal $K$-bundle $\P'$ on $X_\Sigma$ such that $\P$ is ($T$-equivariantly) isomorphic to $\P' \times_K G$. This means that $\P$ has a $T$-equivariant trivializing open cover such that the corresponding transition functions all lie in $K$. Let $\mathcal{P}$ be a toric principal $G$-bundle over $X_\Sigma$. We say that $\P$ \emph{splits equivariantly} if the structure group of $\P$ can be reduced equivariantly to a maximal torus $H \subset G$.  
By Theorem \ref{th-intro-principal-bundle}, after choosing a frame, $\P$ corresponds to a piecewise linear map $\Phi$ from $|\Sigma|$ to $\tilde{\B}(G)$. Suppose the image of $\Phi$ lands in a single apartment (corresponding to a maximal torus $H \subset G$). It follows from the proof of Theorem \ref{th-principal-G-bundle-plm-v2} that all the transition functions of $\P$ (with respect to the open cover by toric affine charts) can be taken to lie in $H$ and hence $\P$ splits.  
Now, consider the case $X_\Sigma = \mathbb{P}^1$.
By the building axioms (Definition \ref{def-building}) we know that any two simplices lie in a common apartment. Since the fan of $\mathbb{P}^1$ has only two rays, we conclude that the image of $\Phi$ lies in a single apartment (corresponding to some maximal torus $H \subset G$) and hence $\P$ splits.
This is the toric principal bundle version of Grothendieck's theorem on splitting of vector bundles over $\mathbb{P}^1$. 


It is well-known that  the equivariant Chow cohomology ring $A_T^*(X_\Sigma)$ can be identified with the algebra of piecewise \emph{polynomial} functions on $\Sigma$ (see \cite{Payne-Chow-coh-toric}). We can immediately recover piecewise polynomial functions representing the equivariant characteristic classes of a toric principal $G$-bundle $\P$ on $X_\Sigma$ from its piecewise linear map $\Phi$. Fix a maximal torus $H \subset G$ and let $p$ be a Weyl group invariant polynomial on the cocharacter lattice of $H$. The polynomial $p$ defines a well-defined function on the cone over the building $\tilde{\B}(G)$ which we also denote by $p$. We have the following (see Section \ref{sec-char-class} and Theorem \ref{th-Chern-Weil-toric-principal-bundle}). 

\begin{THM}   \label{th-intro-char-classes}
The image of a Weyl group invariant polynomial $p$ under the equivariant Chern-Weil homomorphism is given by the piecewise polynomial function $p \circ \Phi$.
\end{THM}


As a special case of Theorem \ref{th-intro-char-classes}, one recovers the equivariant Chern classes of a toric vector bundle. {Each elementary symmetric function $\epsilon_i: \R^r \to \R$} naturally gives an elementary symmetric function $\epsilon_i: \tilde{\B}(\GL(r)) \to \R$, $i=1, \ldots, r$ (see the end of Section \ref{sec-char-class}). The $T$-equivariant Chern classes of a rank $r$ toric vector bundle $\E$ can readily be obtained from its piecewise linear map $\Phi$, that is, the composition $\epsilon_i \circ \Phi: |\Sigma| \to \R$ is exactly the piecewise linear map representing the $i$-th equivariant Chern class of $\E$. This description of equivariant Chern classes in not new and can be found in \cite[Proposition 3.1]{Payne-moduli}). 

\begin{REM}
The present paper is a companion paper to \cite{KM-PL} where the classification of toric principal bundles is far extended to \emph{toric flat families}. One of the main results there states that torus equivariant flat families $\pi: \mathcal{X} \to X_\Sigma$ with reduced and irreducible affine fibers and generic fiber $Y=\Spec(R)$ are classified by \emph{piecewise linear valuations} on $R$. This valuation perspective is then used to obtain results on finite generation of Cox rings of projectivized toric vector bundles. This point of view also opens doors to the study of tropical geometry over the semifield of piecewise linear functions.    
\end{REM}

\begin{REM}
It was suggested to us by Bogdan Ion, that our Theorem \ref{th-intro-principal-bundle} can be interpreted as saying $\tilde{\B}(G)$ is a tropical or piecewise linear analogue of the classifying space of the group $G$. 
\end{REM}

\noindent{\bf Acknowledgment.} We would like to thank Sam Payne, Greg Smith, Indranil Biswas, Arijit Dey, Mainak Poddar, Kelly Jabbusch, Sandra Di Rocco, Roman Fedorov, Bogdan Ion and Dave Anderson for useful conversations and email correspondence. In particular, the authors are grateful to Greg Smith who brought their attention to the interesting subject of toric vector bundles. We also thank Mainak Poddar for suggesting the term {\it framed} toric principal bundle. We also thank Boris Tsvelikhovsky for reading of some parts of the paper and giving useful comments. Finally we are in debt to anonymous referee for numerous valuable suggestions that greatly improved the content and presentation of the paper.
\\

\noindent{\bf Notation.} Throughout the paper we will use the
following notation: 
\begin{itemize}
\item $\k$ is the base field which we take to be algebraically closed.
\item $G$ is a reductive algebraic group over $\k$. 
\item {For a maximal torus $H \subset G$, $\Lambda^\vee(H)$ denotes the cocharacter lattice of $H$ and $\Lambda^\vee_\R(H) = \Lambda^\vee(H) \otimes \R$.}
\item $\Delta(G)$ is the Tits building of $G$ as an abstract simplicial complex. The simplices in $\Delta(G)$ correspond to parabolic subgroups of $G$ and the apartments correspond to the maximal tori in $G$ (see Section \ref{sec-buildings}).
\item {$\B(G)$ is the topological space that is the union of unit spheres in $\Lambda^\vee_\R(H)$, for all maximal tori $H \subset G$ and glued together along simplices corresponding to the same parabolic subgroups (see Definition \ref{def-B}). When $G$ is semisimple, $\B(G)$ is the underlying space of the simplicial complex $\Delta(G)$. By slight abuse of terminology we refer to $\B(G)$ as \emph{(the underlying space of) the Tits building of $G$}.} 
\item {$\tilde{\B}(G)$ is the topological space that is the union of all the vector spaces $\Lambda^\vee_\R(H)$, for all maximal tori $H \subset G$ and glued together along faces of Weyl chambers corresponding to the same parabolic subgroups (see Definition \ref{def-tilde-B}). When $G$ is semisimple, $\tilde{\B}(G)$ is the cone over $\B(G)$. For a reductive group $G$ with semisimple quotient $G_{\textup{ss}}$, $\tilde{\B}(G)$ is the Cartesian product of $\tilde{\B}(G_{\textup{ss}})$ with the real vector space $\Lambda^\vee(Z) \otimes \R$ corresponding to the identity component $Z$ of the center of $G$. By slight abuse of terminology, we refer to $\tilde{\B}(G)$ as the \emph{cone over the Tits building of $G$}.}
\item $E \cong \k^r$ is a finite dimensional $\k$-vector space. It is usually taken to be the fiber over the distinguished point $x_0$ of a toric vector bundle. 
\item $\Delta(\GL(E))$ is the Tits building of $\GL(E)$. Its vertices correspond to flags of subspaces in $E$. Apartments correspond to choices of frames in $E$, that is, direct sum decompositions of $E$ into $1$-dimensional subspaces (see Example \ref{ex-building-GL_r}).  
\item $\tilde{\B}(\GL(E))$ denotes the cone over $\B(\GL(E))$. We can identify $\tilde{\B}(\GL(E))$ with the set of all prevaluations $v: E \setminus \{0\} \to \R$ (see Example \ref{ex-building-GL_r} and Section \ref{subsec-buildings-preval}). 
\item $T \cong \mathbb{G}_m^n$ denotes an algebraic torus with $M$ and $N$ its character and cocharacter lattices respectively. In general, $M$ and $N$ denote rank $n$ free abelian groups dual to each other. We denote the pairing between them by $\langle \cdot, \cdot \rangle: M \times N \to \Z$. We let $M_\R = M \otimes \R$ and $N_\R = N \otimes \R$ to be the corresponding $\R$-vector spaces. 
\item $X_\sigma$ is the affine toric variety corresponding to a (strictly convex rational polyhedral) cone $\sigma \subset N_\R$.
\item $\Sigma$ is a fan in $N_\R$ with corresponding toric variety $X_\Sigma$. We denote the support of $\Sigma$, i.e. the union of cones in it, by $|\Sigma|$. 
\item $\rho$ denotes a ray in $\Sigma$ with ${\bf v}_\rho$ its primitive vector, i.e. the smallest nonzero integral vector along $\rho$.
\item $\Phi: |\Sigma| \to \tilde{\B}(G)$ is a piecewise linear map to the cone over the Tits building of $G$ (see Definition  \ref{def-plm-to-B(G)}).
\end{itemize}

\section{Preliminaries}
\subsection{Preliminaries on toric vector bundles}   \label{sec-prelim-toric-vb}
In this section we review Klyachko's classification of toric vector bundles \cite{Klyachko}.
We mainly follow the exposition in \cite[Section 2]{Payne-moduli}. The first classification of toric vector bundles goes back to \cite{Kaneyama}. We refer the reader to \cite[Section 2.4]{Payne-moduli} for a nice brief history of the subject. 

Let $T \cong \mathbb{G}_m^n$ denote an $n$-dimensional algebraic torus over an algebraically closed field $\k$. We let $M$ and $N$ denote its character and cocharacter lattices respectively. We also denote by $M_\R$ and $N_\R$ the $\R$-vector spaces spanned by $M$ and $N$. For cone $\sigma \in N_\R$ let $M_\sigma$ be the quotient lattice:
$$M_\sigma = M / (\sigma^\perp \cap M).$$
Let $\Sigma$ be a (finite rational polyhedral) fan in $N_\R$ and let $X_\Sigma$ be the corresponding toric variety. Also $X_\sigma$ denotes the invariant affine open subset in $X_\Sigma$ corresponding to a cone $\sigma \in \Sigma$. We denote the support of $\Sigma$, that is the union of all the cones in $\Sigma$, by $|\Sigma|$. For each $i$, $\Sigma(i)$ denotes the subset of $i$-dimensional cones in $\Sigma$. In particular, $\Sigma(1)$ is the set of rays in $\Sigma$. For each ray $\rho \in \Sigma(1)$ we let $\v_\rho$ be the primitive vector along $\rho$, i.e. $\v_\rho$ is the unique vector on $\rho$ whose integral length is equal to $1$.

We say that $\E$ is a {\it toric vector bundle} on $X_\Sigma$ if $\E$ is a vector bundle on $X_\Sigma$ equipped with a $T$-linearization. This means that there is an action of $T$ on $\E$ that lifts the $T$-action on $X_\Sigma$ such that the action map $\E_x \to \E_{t \cdot x}$ for any $t \in T$, $x \in X_\Sigma$ is linear. {By a morphism of toric vector bundles on $X_\Sigma$ we mean a $T$-equivariant morphism of vector bundles.}

We fix a point $x_0 \in X_0 \subset X_\Sigma$ in the dense orbit $X_0$. We often identify $X_0$ with $T$ and think of $x_0$ as the identity element in $T$. We let $E = \E_{x_0}$ denote the fiber of $\E$ over $x_0$. It is an $r$-dimensional vector space where $r = \rank(\E)$. 

For each cone $\sigma \in \Sigma$ we have an invariant open subset $X_\sigma \subset X_\Sigma$. The space of sections $\Gamma(X_\sigma, \E)$ is a $T$-module. We let  $\Gamma(X_\sigma, \E)_u \subseteq\Gamma(X_\sigma, \E)$ be the weight space corresponding to a weight $u \in M$; these spaces define the weight decomposition: 

$$\Gamma(X_\sigma, \E) = \bigoplus_{u \in M} \Gamma(X_\sigma, \E)_u.$$ Every section in $\Gamma(X_\sigma, \E)_u$ is determined by its value at $x_0$.  Thus, by restricting sections to $E = \E_{x_0}$, we get an embedding $\Gamma(X_\sigma, \E)_u \hookrightarrow E$. Let us denote the image of $\Gamma(X_\sigma, \E)_u$ in $E$ by $E_u^\sigma$. Note that if $u' \in \sigma^\vee \cap M$ then multiplication by the character $\chi^{u'}$ gives an injection $\Gamma(X_\sigma, \E)_u \hookrightarrow \Gamma(X_\sigma, \E)_{u-u'}$. Moreover, the multiplication map by $\chi^{u'}$ commutes with the evaluation at $x_0$ and hence induces an inclusion $E_u^\sigma \subset E_{u-u'}^\sigma$. If $u' \in \sigma^\perp$ then these maps are isomorphisms and thus $E_u^\sigma$ depends only on the class $[u] \in M_\sigma = M/(\sigma^\perp \cap M)$. For a ray $\rho \in \Sigma(1)$ we write $$E^\rho_i = E_u^\rho,$$ for any $u \in M$ with $\langle u, \v_\rho \rangle = i$ (all such $u$ define the same class in $M_\rho$). Equivalently, one can define $E^\rho_u$ as follows (see \cite[\S 0.1]{Klyachko}). Pick a point $x_\rho$ in the orbit $O_\rho$ and let:
$$E^\rho_u = \{ e \in E \mid \lim_{t \cdot x_0 \to x_\rho} \chi^u(t)^{-1}(t \cdot e) \textup{ exists in } \E \},$$
where $t$ varies in $T$ in such a way that $t \cdot x_0$ approaches $x_\rho$. One checks that $E^\rho_u$ does not depend on the choice of $x_\rho$ and only depends on $i = \langle u, \v_\rho \rangle$. 

We thus have a decreasing filtration of $E$:
\begin{equation}  \label{equ-filt-E-rho}
\cdots \supset E^\rho_{i-1} \supset E^\rho_i \supset E^\rho_{i+1} \supset \cdots
\end{equation}

An important step in the classification of toric vector bundles is that a toric vector bundle over an affine toric variety is {\it equivariantly trivial}. That is, it decomposes $T$-equivariantly as a sum of trivial line bundles. Let $\sigma$ be a strictly convex rational polyhedral cone with corresponding affine toric variety $X_\sigma$. Given $u \in M$, let $\L_u$ be the trivial line bundle $X_\sigma \times \mathbb{A}^1$ on $X_\sigma$ where $T$ acts on $\mathbb{A}^1$ via the character $u$. One observes that the toric line bundle $\L_u$ in fact only depends on the class $[u] \in M_\sigma$. Hence we also denote this line bundle by $\L_{[u]}$. One has the following:
\begin{proposition}   \label{prop-toric-vb-over-affine-equiv-trivial}
{Let $\E$ be a toric vector bundle of rank $r$ on an affine toric variety $X_\sigma$. Then $\E$ splits equivariantly into a sum of line bundles $\L_{u_j}$: 
$$\E = \bigoplus_{j=1}^r \L_{[u_j]}$$
where $[u_j] \in M_\sigma$.}
\end{proposition}
\begin{proof}
Without loss of generality we can assume that $\sigma$ is a full dimensional cone. Let $x_\sigma$ denote the unique $T$-fixed point in $X_\sigma$. One knows that every vector bundle over an affine variety is globally generated \cite[Example II 5.16.2]{Hartshorne}. Thus the restriction map $H^0(X_\sigma, \E) \to \E_{x_\sigma}$ is a surjective $T$-equivariant map. Hence we can find weight sections $s_1, \ldots, s_r \in H^0(X_\sigma, \E)$ such that their restrictions form a basis for $\E_{x_\sigma}$. Now the set of $x \in X_\sigma$ where the set $\{s_1(x), \ldots, s_r(x)\}$ is linearly dependent is a closed $T$-invariant subset which does not contain $x_\sigma$, hence it must be the empty set. The weight sections $s_i$ then provide an equivariant trivialization of $\E$. 
\end{proof}
We usually denote the multiset $\{ [u_1], \ldots, [u_r]\} \subset M_\sigma$ by $u(\sigma)$. The above shows that the filtrations $(E^\rho_i)_{i \in \Z}$, $\rho \in \Sigma(1)$, satisfy the following compatibility condition: 
{There is a decomposition $E = \bigoplus_{j=1}^r L_j$ of $E$ into a direct sum of $1$-dimensional subspaces $L_j$ and a multiset $u(\sigma) = \{ [u_1], \ldots, [u_r]\} \subset M_\sigma$ such that for any ray $\rho \in \sigma(1)$ we have:
\begin{equation}  \label{equ-Klyachko-comp-condition}
E^\rho_i = \sum_{\langle u_j, \v_\rho \rangle \geq i}  L_j.
\end{equation}
}
{We call a collection of decreasing $\Z$-filtrations $\{(E_i^\rho) \mid \rho \in \Sigma(1)\}$ a \emph{compatible collection of filtrations} if for any $\sigma \in \Sigma$ there is a direct sum decomposition $E=\bigoplus_{j=1}^r L_j$ of $E$ into $1$-dimensional subspaces and a multiset $\{ [u_1], \ldots, [u_r]\} \subset M_\sigma$ such that \eqref{equ-Klyachko-comp-condition} holds. We also need the notion of a morphism between compatible filtrations. Let $E$, $E'$ be finite dimensional vector spaces with compatible collections of filtrations $\{(E^\rho_i) \mid \rho \in \Sigma(1)\}$ and 
$\{({E'}^\rho_i) \mid \rho \in \Sigma(1)\}$ respectively. A \emph{morphism} between these compatible collections is a linear map $F: E \to E'$ such that $F(E_i^\rho) \subset {E'}_i^\rho$, for all $i \in \Z$ and $\rho \in \Sigma(1)$.} 



{The following is the main result in classification of toric vector bundles (see \cite[Theorem 2.2.1]{Klyachko}).
\begin{theorem}[Klyachko]   \label{th-Klyacjko}
The category of toric vector bundles $\E$ on $X_\Sigma$ is naturally equivalent to the category of compatible collections of filtrations on finite dimensional vector spaces $E$. 
\end{theorem}
}

\subsection{Tits Buildings}  \label{sec-buildings}
A \emph{building} is an (abstract) simplicial complex together with a collection of distinguished subcomplexes called \emph{apartments} that satisfy certain axioms. One can think of the notion of building as a ``discretization'' of the notion of symmetric space from Lie theory and differential geometry. While symmetric spaces exist for real and complex groups, buildings can be defined for any linear algebraic group over a field. 

There are two important kinds of buildings: \textit{spherical} buildings and \textit{affine} buildings. Each apartment in a spherical building is a triangulation of a sphere whereas each apartment in an affine building is a triangulation of an affine space. To a linear algebraic group $G$ one associates its \textit{Tits building} which is an example of a spherical building. Similarly, to a linear algebraic group $G$ over a discretely valued field one associates its \textit{Bruhat-Tits building} which is an example of an affine building. In this paper we will only be concerned with Tits buildings.


We denote the Tits building, as a simplicial complex, of a linear algebraic group $G$ by $\Delta(G)$. We recall that the simplices in $\Delta(G)$ correspond to parabolic subgroups of $G$. For parabolic subgroups $P_1$, $P_2$, the simplex corresponding to $P_1$ lies in the boundary of that of $P_2$ if $P_2 \subset P_1$. The apartments in $\Delta(G)$ correspond to maximal tori in $G$. 
{The apartment corresponding to a maximal torus $H$ is the union of all simplices corresponding to parabolic subgroups $P$ that contain $H$.
Let $\Lambda^\vee(H)$ denote the cocharacter lattice of $H$ and put $\Lambda^\vee_\R(H) = \Lambda^\vee(H) \otimes_\Z \R$. The apartment corresponding to $H$ is the \emph{Coxeter complex} of $(G, H)$ which lives in $\Lambda^\vee_\R(H)$ (see Figure \ref{fig-Coxeter-complex-SL3}).} 

\begin{remark}[Reducing to the semisimple case]   \label{rem-semisimple-quotient}
Every parabolic subgroup contains the solvable radical $R(G)$ of $G$. This implies that the building $\Delta(G)$, as a simplicial complex, can be identified with that of its semisimple quotient $G_{\textup{ss}} = G/R(G)$. 
\end{remark}




\begin{figure}  
    \centering
    \scalebox{.5}{
    \begin{tikzpicture}
    \node[circle, draw, minimum size = 5cm] (c) at (0,0){}; 
    \node[circle, draw=black, fill=black] (c) at (2.5, 0){};
    \node[circle, draw=black, fill=black] (c) at (1.25, 2.165){};
    \node[circle, draw=black, fill=black] (c) at (-1.25, 2.165){};
    \node[circle, draw=black, fill=black] (c) at (-2.5, 0){};
    \node[circle, draw=black, fill=black] (c) at (-1.25, -2.165){};
    \node[circle, draw=black, fill=black] (c) at (1.25, -2.165){};
    
    \draw [-stealth](0,0) -- ($(0,0)!4cm!(2.5,0)$);
    \draw [-stealth](0,0) -- ($(0,0)!4cm!(1.25, 2.165)$);
    \draw [-stealth](0,0) -- ($(0,0)!4cm!(-1.25, 2.165)$);
    \draw [-stealth](0,0) -- ($(0,0)!4cm!(-2.5, 0)$);
    \draw [-stealth](0,0) -- ($(0,0)!4cm!(-1.25, -2.165)$);
    \draw [-stealth](0,0) -- ($(0,0)!4cm!(1.25, -2.165)$);
    
    \end{tikzpicture}}
\caption{An apartment in the Tits building of $\SL(3)$. It is a triangulation of the circle.}  
\label{fig-Coxeter-complex-SL3}
\end{figure}

By Remark \ref{rem-semisimple-quotient}, {without changing the building (as an abstract simplicial complex)}, we can always reduce to the semisimple case. Let $G$ be a semisimple linear algebraic group. The abstract simplicial complex $\Delta(G)$ has a natural geometric realization, that is, a topological space $\B(G)$ such that each simplex in $\Delta(G)$ can be identified with a subset of $\B(G)$ homeomorphic to a standard simplex and these simplices intersect, as subsets of $\B(G)$, along their common subsimplices. Below we explain the construction of $\B(G)$.

Let $H$ be a maximal torus and let $A_H$ denote the sphere obtained by taking the quotient of $\Lambda_\R^\vee(H) \setminus \{0\}$ by the positive scalars $\R_{>0}$. If we fix an inner product on $\Lambda_\R^\vee(H)$, this quotient can be identified with the unit sphere in $\Lambda_\R^\vee(H)$. {When $G$ is semisimple, the Weyl chambers are simplicial cones and the Coxeter complex of $(G, H)$ can be realized as the triangulation of $A_H$ obtained by the images of Weyl chambers (see Figure \ref{fig-Coxeter-complex-SL3}).}

There is a one-to-one correspondence between the faces of Weyl chambers in $\Lambda_\R^\vee(H)$ and the parabolic subgroups containing $H$. Suppose a parabolic subgroup $P$ has two maximal tori $H$ and $H'$ respectively. 
In Section \ref{subsec-1-para-subgp} (Proposition  \ref{prop-P-lambda}) we will see that 
there is a natural linear isomorphism between the face in $\Lambda_\R^\vee(H)$ corresponding to $P$ and the one in $\Lambda_\R^\vee(H')$ corresponding to $P$. Moreover, this linear isomorphism sends lattice points to lattice points (i.e. sends the points in $\Lambda^\vee(H)$ to the points in $\Lambda^\vee(H')$).

\begin{definition}[Underlying space of the Tits building]
\label{def-B}
{Let $G$ be a linear algebraic group. Consider the topological space $\B(G)$ obtained by gluing the spheres $A_H$, for all maximal tori $H \subset G$, along the simplices corresponding to the same parabolic subgroups. We remark that when $G$ is semisimple, the topological space $\B(G)$ is the underlying space of the simplicial complex $\Delta(G)$, that is, every simplex in $\Delta(G)$ corresponds to a subset of $\B(G)$ homeomorphic to a standard simplex. 

By slight abuse of terminology, for a linear algebraic group $G$, we refer to $\B(G)$ as the \emph{underlying space of the Tits building of $G$} (or just the \emph{Tits building of $G$} for short). We also refer to $A_H $ as the \emph{underlying space of the apartment associated to $H$}.}
\end{definition}

\begin{definition}[Cone over the Tits building]
\label{def-tilde-B}
{
Let $G$ be a linear algebraic group. We let $\tilde{\B}(G)$ to be the topological space obtained by gluing the vector spaces $\Lambda_\R^\vee(H)$, for all maximal tori $H \subset G$, along faces of Weyl chambers corresponding to the same parabolic subgroups. When $G$ is semisimple, $\tilde{\B}(G)$ is the cone over $\B(G)$. 
If $G$ is a reductive group with semisimple quotient $G_{\textup{ss}}$, $\tilde{\B}(G)$ is the Cartesian product of $\tilde{\B}(G_{\textup{ss}})$ with the real vector space $\Lambda^\vee(Z) \otimes \R$, where $Z$ denotes the identity component in the center of $G$.

By slight abuse of terminology, for a linear algebraic group $G$, we refer to $\tilde{\B}(G)$ as the \emph{cone over the Tits building of $G$}.
For a maximal torus $H$, we denote the vector space  $\Lambda_\R^\vee(H)$ by $\tilde{A}_H$ and refer to it as the \emph{cone over the apartment $A_H$}.
We denote the subset of $\tilde{\B}(G)$ obtained by taking the union of lattices $\Lambda^\vee(H)$, for all maximal tori $H$, by $\tilde{\B}_\Z(G)$ and refer to it as the \emph{set of lattice points} in $\tilde{\B}(G)$.} 
\end{definition}

In this paper we only deal with the Tits building associated to a linear algebraic group and we are not concerned with the general theory of buildings. Nevertheless we recall the definition of a building, as an abstract simplicial complex, here. 
\begin{definition}[Building as an abstract simplicial complex]  \label{def-building}
A $d$-dimensional {\it building} $\Delta$ is an abstract simplicial complex that is a union of subcomplexes $A$ called {\it apartments} satisfying the following axioms:
\begin{itemize}
\item[(a)] Every $k$-simplex of $\Delta$ is within at least three $d$-simplices if $k < d$.
\item[(b)] Any $(d-1)$-simplex in an apartment $A$ lies in exactly two adjacent $d$-simplices of $A$ and the graph of adjacent $d$-simplices is connected.
\item[(c)] Any two simplices in $\Delta$ lie in some common apartment $A$. If two simplices both lie in apartments $A$ and $A'$, then there is a simplicial isomorphism of $A$ onto $A'$ fixing the vertices of the two simplices.
\end{itemize} 
A $d$-simplex in $A$ is called a {\it chamber}. The rank of the building is defined to be $d + 1$.
\end{definition}

\subsection{Tits buildings and one-parameter subgroups}   \label{subsec-1-para-subgp}
We present a natural way to realize (the cone over) the Tits building of $G$, namely, as the space of all algebraic one-parameter subgroups of $G$ modulo a certain equivalence relation. This construction of the Tits building of a linear algebraic group from one-parameter subgroups also appears in \cite[Section 2.2]{Mumford}.
The difference between our approach in this section and that of \cite[Section 2.2]{Mumford} is that the latter is interested in describing the Tits building whereas we are interested in describing the \emph{cone} over the Tits building. 

Let $G$ be a linear algebraic group. By an \emph{algebraic one-parameter subgroup} of $G$ (or a \emph{one-parameter subgroup} for short) we mean a homomorphism of algebraic groups $\lambda: \mathbb{G}_m \to G$. 


\begin{definition}[Equivalence of one-parameter subgroups]   \label{def-1-para-equiv}
Let $\lambda_1$, $\lambda_2$ be one-parameter subgroups of $G$. We say that $\lambda_1$ and $\lambda_2$ are \emph{equivalent} and write $\lambda_1 \sim \lambda_2$ if the the function $\lambda_1\lambda_2^{-1}: \G_m \to G$ extends to a regular function from $\mathbb{A}^1 $ to $G$, that is, if $\lim_{s \to 0} \lambda_1(s)\lambda_2(s)^{-1}$ exists in $G$. 
\end{definition}

We will see (Corollary \ref{cor-tilde-B-one-para}) that the equivalence classes of $\sim$ can be identified with $\tilde{\B}_\Z(G)$, the set of lattice points in the cone over the Tits building $\tilde{\B}(G)$ (Definition \ref{def-tilde-B}).

To a one-parameter subgroup $\lambda: \G_m \to G$ there corresponds a parabolic subgroup $P_\lambda \subset G$ defined as follows (see \cite[Section 2.2, Definition 2.3/Proposition 2.6]{Mumford}, \cite[Chapter 21, Section d]{Milne}):
$$P_\lambda = \{ g \in G \mid \lim_{s \to 0} \lambda(s) g \lambda(s)^{-1} \textup{ exists in }  G \}.$$

Fix a faithful representation $G \hookrightarrow \GL(E)$ where $E$ is a finite dimensional vector space. The one-parameter subgroup $\lambda$ then gives a $\G_m$-action on $E$. Let $E = \bigoplus_{i=1}^k W_i$ be the weight space decomposition of $E$ and let $c_i \in \Z$ be the weight of the weight space $W_i$. We assume $c_1 > \cdots > c_k$. This gives rise to a flag of subspaces $$F_\bullet = (\{0\} \subset F_1 \subset \cdots \subset F_k = E),$$ where $F_j = \bigoplus_{i=0}^j W_i$.
The stabilizer of the flag $F_\bullet$ in $\GL(E)$ is a parabolic subgroup of $\GL(E)$. One can show that the parabolic subgroup $P_\lambda$ is the stabilizer of $F_\bullet$ in $G$. The following proposition gives alternative characterizations of the equivalence of one-parameter subgroups. Item (c) in Proposition \ref{prop-1-para-equiv} is used  in \cite[Section 2.2, Definition 2.5]{Mumford} to define a variant of the equivalence $\sim$ of one-parameter subgroups.

\begin{proposition}    \label{prop-1-para-equiv}
Let $\lambda_1$, $\lambda_2$ be one-parameter subgroups of $G$. Fix a faithful representation $G \hookrightarrow \GL(E)$. For $i=1, 2$, let $F_{i, \bullet}$ and $c^i_1 > \cdots > c^i_k$ be the flag and weights associated to the linear action of $\lambda_i$ on $E$ respectively. The following are equivalent:
\begin{itemize}
\item[(a)] $\lambda_1 \sim \lambda_2$.
\item[(b)] $\lambda_1$ and $\lambda_2$ have the same flags and the set of weights, that is, $F_{1, \bullet} = F_{2, \bullet}$ and  $c^1_j = c^2_j$, $j=1, \ldots, k$. 
\item[(c)]	There exists $g \in P_{\lambda_1}$ such that $\lambda_2 = g \lambda_1 g^{-1}$.
\end{itemize}
\end{proposition}   
We need the following lemma in the proof of Proposition \ref{prop-1-para-equiv}.
\begin{lemma}    \label{lem-1-para-subgp}
Let $\lambda: \G_m \to \GL(E)$ be a one-parameter subgroup with the corresponding flag $F_\bullet$ and let $\lambda$ be diagonal in a basis $B$ for $E$. Let $\lambda'$ be another one-parameter subgroup which fixes the flag $F_\bullet$. Then there is $x \in P_\lambda$ such that $x\lambda'x^{-1}$ is diagonal in the basis $B$.	
\end{lemma}
\begin{proof}
For each $j=1, \ldots, k$, let $B_j = B \cap F_j$, where $F_j$ is the $j$-th subspace in the flag. Then $B_j$ is a basis for $F_j$. We know $\lambda'$ fixes every subspace $F_j$. Let $B_1'$ be a basis for $F_1$ consisting of weight vectors for $\lambda'$. Extend $B_1'$ to a basis $B_2'$ for $F_2$ consisting of weight vectors for $\lambda'$. Continuing this way we arrive at a basis $B'$ for $E$ of weight vectors for $\lambda'$. Now let $x \in \GL(E)$ be a linear transformation which sends $B'_j$ to $B_j$, $j=1, \ldots, k$. Then $x \in P_\lambda$ because it fixes the flag $F_\bullet$. One sees from the construction that $x\lambda'x^{-1}$ is diagonal in the basis $B$.
\end{proof}
\begin{proof}[Proof of Proposition \ref{prop-1-para-equiv}]
{It is well-known that given two flags in $E$ one can find a basis $B$ that is \emph{adapted} to both of the flags, that is, any subspace appearing in either flag is spanned by a subset of $B$. This is in fact one of the axioms in the definition of a building for the Tits building of $\GL(E)$ (see Definition  \ref{def-building}(c) and Example \ref{ex-building-GL_r}) and can be deduced as a corollary of the proof of Jordan-H\"older theorem (see \cite[Section 4.3]{Abramenko-Brown}).}
On the other hand, by Lemma  \ref{lem-1-para-subgp} we can find $x_i \in P_{\lambda_i}$, $i=1, 2$, such that both $\lambda_i' = x_i \lambda_i x_i^{-1}$ are diagonal in the basis $B$. Now, for $i=1, 2$, since $x_i \in P_{\lambda_i}$, we have $\lim_{t \to 0} \lambda_i(t) x_i \lambda_i(t)^{-1}$ exists. This implies that $\lim_{t \to 0} \lambda_i(t) \lambda'_i(t)^{-1}$ also exists which means $\lambda_i \sim \lambda_i'$. Thus, without loss of generality, we can replace $\lambda_i$ with $\lambda_i'$. Since $\lambda'_1$ and $\lambda'_2$ are diagonal in the same basis, one observes that $\lambda'_1 \sim \lambda'_2$ if and only if $\lambda'_1=\lambda'_2$. 
The proposition is straightforward to verify for $\lambda'_1$ and $\lambda'_2$. That is, $\lambda'_1 = \lambda'_2$ if and only if they have the same flags and the same set of weights, if and only if $\lambda'_2 = g \lambda'_1 g^{-1}$ for some $g \in P_{\lambda'_1}$. This concludes the proof. 
\end{proof}

Let $H \subset G$ be a maximal torus. It is well-known that the parabolic subgroups $P$ containing $H$ are in one-to-one correspondence with the faces of Weyl chambers in the vector space $\Lambda^\vee_\R(H)$. For a parabolic subgroup $P \subset G$ we let $\Lambda^\vee_P(H)$ denote the set of one-parameter subgroups in $\Lambda^\vee(H)$ that lie on the face  corresponding to $P$. That is, $\Lambda^\vee_P(H)$ consists of one-parameter subgroups in $\Lambda^\vee(H)$ that act on $\Lie(P)$ with non-negative weights. Alternatively, we have:
$$\Lambda^\vee_P(H) = \{ \lambda \in \Lambda^\vee(H) \mid P_\lambda = P \}.$$

Given two maximal tori $H$, $H'$ in $P$, the next proposition gives a natural linear isomorphism between the faces of Weyl chambers in $\Lambda^\vee_\R(H)$ and $\Lambda^\vee_\R(H')$ corresponding to the parabolic subgroup $P$. 
\begin{proposition}  \label{prop-P-lambda}
Let $P \subset G$ be a parabolic subgroup. 
\begin{itemize}
\item[(a)] Let $H$, $H'$ be maximal tori in $P$. Then for every $\lambda \in \Lambda^\vee_P(H)$ there is a unique $\lambda' \in \Lambda^\vee(H')$ such that $\lambda \sim \lambda'$.  	
\item[(b)] Moreover, $\lambda \mapsto \lambda'$ extends to a linear isomorphism from the cone associated to $P$ in $\Lambda_\R^\vee(H)$ to the cone associated to $P$ in $\Lambda_\R^\vee(H')$.
\end{itemize}
\end{proposition}
\begin{proof}	
(a) Since maximal tori are conjugate, there exists $g \in P = P_\lambda$ such that $gHg^{-1} = H'$. Let $\lambda' = g \lambda g^{-1}$. Then $\lambda' \in \Lambda^\vee(H')$ and $\lambda \sim \lambda'$ by Proposition  \ref{prop-1-para-equiv}(c). Also, no two one-parameter subgroups in $\Lambda^\vee(H')$ are equivalent which proves the uniqueness of $\lambda'$. (b) Clearly $\lambda \mapsto g \lambda g^{-1}$ is a group isomorphism from $\Lambda^\vee(H)$ to $\Lambda^\vee(H')$ and maps $\Lambda^\vee_P(H)$ onto $\Lambda^\vee_P(H')$.	
\end{proof}

For a parabolic subgroup $P$, let $\Lambda^\vee_P \subset \tilde{\B}_\Z(G)$ be the set of equivalence classes of one-parameter subgroups of $G$ that act on $\Lie(P)$ by non-negative weights. Alternatively:
$$ \Lambda^\vee_P = \{ \lambda \mid P_\lambda = P\} /  \sim.$$
It follows from Proposition \ref{prop-P-lambda} that, for each maximal torus $H \subset P$, the set $\Lambda^\vee_P$ can be identified with $\Lambda^\vee_P(H)$, the set of lattice points in the face associated to $P$.

The following is the main conclusion of this section and is an immediate consequence of Proposition \ref{prop-P-lambda}.
\begin{corollary}   \label{cor-tilde-B-one-para}
The set $\tilde{\B}_\Z(G)$ of lattice points in the cone over the Tits building (Definition \ref{def-tilde-B}) can be identified with the set of equivalence classes of  one-parameter subgroups of $G$. 
\end{corollary}

\begin{remark}
The above realization of the Tits building of $G$ in terms of equivalence classes of one-parameter subgroups (Corollary \ref{cor-tilde-B-one-para}) is analogous to the description of the Tits building of a symmetric space as the set of equivalence classes of geodesics (see \cite[Section 3]{Ji}).
\end{remark}







A homomorphism of algebraic groups induces a map between the corresponding cones over the Tits buildings. Realizing {the set of lattice points in} the cone over the Tits building as a quotient of space of one-parameter subgroups gives an easy way to construct this map. 
\begin{definition}   \label{def-B(G)-B(G')}
Let $\alpha: G \to G'$ be a homomorphism of linear algebraic groups. If $\lambda: \G_m \to G$ is a one-parameter subgroup of $G$ then $\alpha \circ \lambda: \G_m \to G'$ is a one-parameter subgroup of $G'$. It follows from the definition that if $\lambda \sim \lambda'$ then $\alpha \circ \lambda \sim \alpha \circ \lambda'$. Hence $\lambda \mapsto \alpha \circ \lambda$ gives a well-defined map $\hat{\alpha}: \tilde{\B}(G) \to \tilde{\B}(G')$. If $H \subset G$ is a maximal torus then $\alpha(H)$ is contained in a maximal torus $H'$ of $G'$ and hence $\hat{\alpha}(\tilde{A}_H) \subset \tilde{A}_{H'}$. 
\end{definition}

\subsection{Examples}
In this section we describe (the cones over) the Tits buildings of the general linear group, the symplectic group and the even orthogonal group.
\begin{example}[Tits building of the general linear group]   \label{ex-building-GL_r}
Let $E$ be a $\k$-vector space of dimension $r$. The (cone over the) Tits building of $G=\GL(E)$ has a nice concrete description as follows. Each one-parameter subgroup $\lambda: \mathbb{G}_m \to \GL(E)$ is a diagonal matrix $\operatorname{diag}(t^{v_1}, \ldots, t^{v_r})$, $v_i \in \Z$, in some basis $B = \{b_1, \ldots, b_r\}$ for $E$. After reordering the basis elements if necessary we can assume $v_1 \geq \cdots \geq v_r$. The ordered basis $B$ gives rise to a complete flag of subspaces 
$$V_\bullet = (\{0\} \subsetneqq V_1 \subsetneqq \cdots \subsetneqq V_r = E),$$ 
where $V_i = \operatorname{span}\{b_1, \ldots, b_i\}$. 
Suppose
$$v_1 = \cdots = v_{i_1} > v_{i_1+1} = \cdots = v_{i_2} > \cdots  =v_{i_{k-1}} > v_{i_{k-1}+1} = \cdots = v_{i_k} = v_{r}.$$
Thus, $v_{i_1} > \cdots > v_{i_k} = v_r$ are the distinct numbers among the $v_i$. For $j=1, \ldots, k$, let $c_j = v_{i_j}$ and $F_j = V_{i_j}$ and put $c_\bullet =(c_1 > \cdots > c_k)$ and $F_\bullet = (\{0\} \subsetneqq F_1 \subsetneqq \cdots \subsetneqq F_k = E)$. We conclude that a one-parameter subgroup in $\GL(E)$ is uniquely determined by the pair $(F_\bullet, c_\bullet)$. Let us call $(F_\bullet, c_\bullet)$, where $c_\bullet = (c_1 > \cdots > c_k)$ is a decreasing sequence of real numbers, a {\it labeled flag}. The cone over the building $\tilde{\B}(\GL(E))$ can be realized as the collection of labeled flags in $E$. 

\begin{remark}  \label{rem-bldg-GL(r)-filtrations}
{Alternatively, we can think of the integral labeled flags (i.e. the points in $\tilde{\B}_\Z(\GL(E)) \subset \tilde{\B}(\GL(E))$)} as decreasing $\Z$-filtrations $(E_i)$ in $E$. Namely, given an integral labeled flag $(F_\bullet, c_\bullet)$ the corresponding filtration $(E_i)$ is defined as follows: we let $E_{i} = F_j$ whenever $c_{j} \geq i > c_{j+1}$, $j=1, \ldots, k-1$. If $c_k \geq i$ we let $E_{i} = E$, and if $i > c_1$ we let $E_{i} = \{0\}$. 
\end{remark}

A {\it frame} in $E$ is a direct sum decomposition $$E = \bigoplus_{i=1}^r L_i$$ into one-dimensional subspaces. Maximal tori in $\GL(E)$ are in one-to-one correspondence with frames in $E$. For a frame $L = \{L_1, \ldots, L_r\}$ giving a maximal torus $H$, the corresponding cone over the apartment $\tilde{A}_L = \tilde{A}_H$ is the $r$-dimensional $\R$-vector space of all functions from the finite set $\{L_1, \ldots, L_r\}$ to $\R$. We say that a flag $F_\bullet$ is {\it adapted} to a frame $L$ if every subspace $F_i$ in it is spanned by some of the one-dimensional subspaces in $L$.

The abstract simplicial complex, $\Delta(\GL(E))$ can be described as follows. Its vertices are the proper nontrivial vector subspaces of $E$. Two subspaces $U_1$ and $U_2$ are connected if one of them is a subset of the other. The $m$-simplices of $\Delta(\GL(E))$ are formed by sets of $m + 1$ mutually connected subspaces, namely a flag $F_\bullet = (\{0\} \subsetneqq F_1 \subsetneqq \cdots \subsetneqq F_{m+1} \subsetneqq F_{m+2} = E)$. Maximal connectivity is obtained by taking $r - 1$ proper nontrivial subspaces and the corresponding simplex corresponds to a complete flag.
\end{example}

\begin{example}[Tits buildings of the orthogonal and symplectic groups]   \label{ex-building-orth-symp}
Next let us describe (the cones over) the Tits buildings of the symplectic group and the even orthogonal group. The odd orthogonal group can be treated in a similar fashion. 
  
Let $E$ be a $2r$-dimensional vector space over $\k$. Let $\langle \cdot, \cdot \rangle$ be a non-degenerate symmetric or skew-symmetric bilinear form on $E$. When $\langle \cdot, \cdot \rangle$ is symmetric (respectively skew-symmetric) denote the subgroup of $\GL(E)$ preserving $\langle \cdot, \cdot \rangle$ by $\operatorname{O}(E)$ (respectively $\operatorname{Sp}(E)$). Throughout the rest of this example, $G$ denotes $\operatorname{O}(E)$ or $\operatorname{Sp}(E)$. 

Consider a partial flag of subspaces 
$$F_\bullet = (\{0\} \subsetneqq F_1 \subsetneqq \cdots \subsetneqq F_{k} = E).$$
We call $F$ an {\it isotropic flag} if for each $0 \leq j \leq k$ we have 
\begin{equation}   \label{equ-isotropy}
F_j^\perp = F_{k - j}.
\end{equation}
This implies that $F_j$ is an isotropic subspace, for $j=1, \ldots, \lfloor k/2 \rfloor$. Note that the $F_j$, $j=1, \ldots, \lfloor k/2 \rfloor$, determine the whole isotropic flag $F_\bullet$. We mention that in the literature sometimes an isotopic flag is defined to be a flag consisting of isotropic subspaces, that is, $F_1 \subsetneqq \cdots \subsetneqq F_{\lfloor k/2 \rfloor}$ in our notation. 

From linear algebra one knows that there exists a basis $B = \{e_1, \ldots, e_r, f_1, \ldots, f_r\}$ for $E$ such that:
\begin{align*}   \label{equ-normal-basis}
\langle e_i, e_j \rangle &= 0, ~\forall i, j \\
\langle e_i, f_i \rangle &= 1,  ~\forall i \\
\langle e_i, f_j \rangle &= 0, ~\forall i, j, ~ i \neq j. \\
\end{align*}
We call such a basis $B$ a {\it normal basis} for $(E, \langle \cdot, \cdot \rangle)$. Note that for any nonzero $t_1, \ldots, t_n \in \k$, $\{t_1 e_1, \ldots, t_r e_r, t_1^{-1} f_1, \ldots, t_r^{-1} f_r\}$ is also a normal basis. We call a normal basis up to multiplication by the $t_i$ a {\it normal frame}. The normal frames are in one-to-one correspondence with maximal tori of $G$. In fact, an ordered normal basis $(e_1, \ldots, e_r, f_r, \ldots, f_1)$ gives the maximal torus $T$ defined by:
$$T = \{ \operatorname{diag}(t_1, \ldots, t_r, t_r^{-1}, \ldots, t_1^{-1}) \mid t_1, \ldots, t_r \in \k^* \}.$$
We note that one can reorder $\{e_1, \ldots, e_n\}$ as well as switch any $e_i$ with $\pm f_i$ (depending on whether $\langle \cdot, \cdot \rangle$ is symmetric or skew-symmetric) and still have a normal basis. 
Each one-parameter subgroup $\lambda: \G_m \to G$ is given by a diagonal matrix $$\diag(t^{v_1}, \ldots, t^{v_r}, t^{-v_r}, \ldots, t^{-v_1}), ~~ v_i \in \Z,$$ in some ordered normal basis $B = (e_1, \ldots, e_r, f_r, \ldots, f_1)$. After reordering $v_1, \ldots, v_r$ and switching some of the $v_i$ with $-v_i$ if necessary, we can assume $$v_1 \geq \cdots \geq v_r \geq 0 \geq -v_r \geq \cdots \geq -v_1.$$ To simplify the notation, for $i=1, \ldots r$ we let $v_{r+i} = -v_{r+1-i}$.    
The ordered basis $B$ gives rise to a complete flag $(\{0\} \subset V_1 \subset \cdots \subset V_{2r} = E)$ defined by:
\begin{align*}
V_i &= \operatorname{span}\{e_1, \ldots, e_i \} \\
V_{r + i}  &= \operatorname{span}\{e_1, \ldots, e_r, f_r, \ldots, f_{r+1-i}.\}, 
\end{align*}
where $i = 1, \ldots, r$. One observes that $V_i^{\perp} = V_{2r-i}$. 
Let the indices $1 \leq i_1 < \ldots < i_k = 2r$ be such that:
$$v_1 = \cdots = v_{i_1} > v_{i_1+1} = \cdots = v_{i_2} > \cdots v_{i_{k-1}} > v_{i_{k-1}+1} = \cdots v_{i_k} = -v_{1}.$$ That is, $v_{i_1} > \cdots > v_{i_k}$ are the distinct numbers among the $\pm v_i$. For $j=1, \ldots, k$ put $c_j = v_{i_j}$ and $F_j = V_{i_j}$. We thus get an isotropic flag $F_\bullet = (\{0\} \subsetneqq F_1 \subsetneqq \cdots \subsetneqq F_{k} = E)$ and a labeling $c_\bullet = (c_1 > \cdots > c_{k})$ with $c_j = -c_{k+1-j}$. Let us call $(F_\bullet, c_\bullet)$, where $c_\bullet = (c_1 > \cdots > c_{k})$ is a decreasing sequence of real numbers with $c_{j}=-c_{k+1-j}$, a {\it labeled isotropic flag}. The cone over the building $\tilde{\B}(G)$ can be realized as the collection of labeled isotropic flags. We say that an isotropic flag $F_\bullet$ is {\it adapted} to a normal frame $L$ if every subspace $F_i$ in it is spanned by some of the one-dimensional subspaces in $L$.

The abstract simplicial complex $\Delta(G)$ can be described in a similar fashion to that of $\GL(E)$.
The simplices in $\Delta(G)$ correspond to isotropic flags $F_\bullet$ and apartments correspond to normal frames. A simplex is contained in an apartment if the corresponding isotropic flag is adapted to the corresponding normal basis.

{Alternatively, in a similar way as explained in Remark \ref{rem-bldg-GL(r)-filtrations}, the integral labeled isotropic flags (i.e. the points in $\tilde{\B}_\Z(G) \subset \tilde{\B}(G)$) can be realized as decreasing $\Z$-filtrations on $E$ consisting of isotropic subspaces.} 
\end{example}

\subsection{Tits building of the general linear group and prevaluations}  \label{subsec-buildings-preval} 
A useful realization of the (cone over the) Tits building of $\GL(E)$ is as the space of all \emph{prevaluations} on $E$. The authors have not been able to find a reference for this realization in the literature. This is an analogue of the Goldman-Iwahori realization of the \emph{Bruhat-Tits building} of $\GL(E)$ as the space of non-Archimedean norms on $E$ (see \cite[Section 1.2]{RTW} or \cite{JSY}). We recall that to a linear algebraic group $G$ over a discretely valued field, there corresponds an affine building called its Bruhat-Tits building.

\begin{definition}
Let $E$ be a finite dimensional $\k$-vector space.
We call a function $v: E\setminus \{0\} \to \R$ a {\it prevaluation} if the following hold:
\begin{itemize}
\item[(1)] For all $0 \neq e \in E$ and $0 \neq c \in \k$ we have $v(ce) = v(e)$. 
\item[(2)](Non-Archimedean property) For all $0 \neq e_1, e_2 \in E$, $e_1+e_2 \neq 0$, the non-Archimedean inequality $v(e_1+e_2) \geq \min\{v(e_1), v(e_2)\}$ holds. 
\end{itemize}
\end{definition}
It is convenient to extend $v$ to the whole $E$ and define $v(0) = \infty$. We call a prevaluation $v$ {\it integral} if it attains only integer values, i.e. $v: E \setminus \{0\} \to \Z$. The term {\it prevaluation} is taken from the paper \cite[Section 2.1]{KKh-Annals}.  Prevaluation is a standard commutative algebra notion although in most of the literature the term \textit{valuation on a vector space} is used. We use the term prevaluation to distinguish it from valuations on rings. 

The {\it value set} $v(E)$ of a prevaluation $v$ is the image of $E \setminus \{0\}$ under $v$, i.e.
$$v(E) = \{ v(e) \mid 0 \neq e \in E\}.$$ It is easy to verify that $|v(E)| \leq \dim(E)$ and hence $v(E)$ is finite. Each integral prevaluation $v$ on $E$ gives rise to a filtration $E_v = (E_{v \geq a})_{a \in \Z}$ on $E$ by vector subspaces defined by: 
$$E_{v \geq a} = \{ e \in E \mid v(e) \geq a\}.$$
If $v(E) = \{a_1 > \cdots > a_k\}$ then we have an integral labeled flag $F_\bullet = (F_{v \geq a_1} \subsetneqq \cdots \subsetneqq F_{v \geq a_k})$ where $F_{v \geq a_i}$ is labeled with $a_i$. 
Conversely, a decreasing filtration $E_\bullet = (E_a)_{a \in \Z}$ such that $\bigcap_{a \in \Z} E_a = \{0\}$ defines a prevaluation $v_{E_\bullet}$ by:
$$v_{E_\bullet}(e) = \max\{ a \in \Z \mid e \in E_a\},$$ for all $e \in E$. 
The assignments $v \mapsto E_v$ and $E_\bullet \mapsto v_{E_\bullet}$ are inverse of each other and give a one-to-one correspondence between the set of integral prevaluations on $E$ and the set of decreasing  $\Z$-filtrations on $E$.

Recall that a frame $L = \{L_1, \ldots, L_r\}$ is a collection of $1$-dimensional subspaces $L_i$ such that $E = \bigoplus_{i=1}^r L_i$. We say that a frame $L$ is {\it adapted} to a prevaluation $v$ if every subspace $F_{v \geq a}$ is a sum of some of the $L_i$. This is equivalent to the following: For any $0 \neq e \in E$ let us write $e = \sum_i e_i$ where $e_i \in L_i$. Then:
$$v(e) = \min\{ v(e_i) \mid e_i \neq 0 \}.$$

The set of prevaluations $v: E \setminus \{0\} \to \R$ can naturally be identified with $\tilde{\B}(\GL(E))$, the cone over the building of $\GL(E)$. 
Given a frame $L$ the corresponding cone over the apartment $\tilde{A}_L$ consists of all prevaluations adapted to $L$.

\section{Toric principal bundles and Tits buildings}
\label{sec-G-bundle-PLM}
Let $G$ be a linear algebraic group over $\k$. As usual we denote the cone over the Tits building of $G$ by $\tilde{\B}(G)$. 
Recall that for each maximal torus $H \subset G$ we identify the corresponding cone over the apartment $\tilde{A}_H$ with the $\R$-vector space $\Lambda^\vee_\R(H) = \Lambda^\vee(H) \otimes \R$, where $\Lambda^\vee(H)$ denotes the cocharacter lattice of $H$. 

Throughout, $\Sigma$ is a fan in an $\R$-vector space $N_\R$ with support $|\Sigma|$. 

\begin{definition}[Piecewise linear map to $\tilde{\B}(G)$]   \label{def-plm-to-B(G)}
We say that a map $\Phi: |\Sigma| \to \tilde{\B}(G)$ is a {\it piecewise linear map} if the following hold:
\begin{itemize}
\item[(1)] For each cone $\sigma \in \Sigma$ the image $\Phi(\sigma)$ lies in a cone over an apartment $\tilde{A}_\sigma = \Lambda_\R^\vee(H_\sigma)$ (which of course is not necessarily unique). Here $H_\sigma \subset G$ is the maximal torus corresponding to $\tilde{A}_\sigma$. 
\item[(2)] For each $\sigma \in \Sigma$ the restriction 
$\Phi_\sigma := \Phi_{|\sigma}: \sigma \to \tilde{A}_\sigma$ is an $\R$-linear map.
\end{itemize}
We say that $\Phi$ is {\it integral} if for every cone $\sigma$, the map $\Phi_\sigma$ restricts to give a $\Z$-linear map $\Phi_\sigma: \sigma \cap N \to \Lambda^\vee(H_\sigma)$. We note that $N = \Lambda^\vee(T)$ is the cocharacter lattice of $T$ and hence such a linear map is in fact the derivative of a group homomorphism $\phi_\sigma: T_\sigma \to H_\sigma$ where $T_\sigma \subset T$ is the subtorus whose cocharacter lattice is generated by $\sigma \cap N$.
\end{definition}

Let $\alpha: G \to G'$ be a homomorphism of linear algebraic groups. Recall that $\alpha$ induces a map $\hat{\alpha}: \tilde{\B}(G) \to \tilde{\B}(G')$ (see Definition \ref{def-B(G)-B(G')}). If $\Phi: |\Sigma| \to \tilde{\B}(G)$ is a piecewise linear map, it is clear that $\Phi' = \hat{\alpha} \circ \Phi: |\Sigma| \to \tilde{\B}(G')$ is a piecewise linear map as well. 


Recall that a principal $G$-bundle over a variety $X$ is a fiber bundle $\P$ over $X$ with an action of $G$ such that $G$ maps each fiber to itself and the action of $G$ on each fiber is {free and transitive}, that is, every fiber can be considered as a copy of $G$. We will write the action of $G$ on $\P$ as a right action. Let $G$, $G'$ be algebraic groups. Let $\P$ (respectively $\P'$) be a principal $G$-bundle (respectively $G'$-bundle) over $X$. A {\it morphism of principal bundles} is a bundle map $F: \P \to \P'$ with respect to a homomorphism $\alpha: G \to G'$ if the following equivariance condition holds: For any $z \in \P$ and $g \in G$,
$$F(z \cdot g) = F(z) \cdot \alpha(g).$$

With our usual notation, let $X_\Sigma$ be the toric variety associated to a fan $\Sigma$. Following \cite{Biswas} we say that a principal $G$-bundle $\P$ over a toric variety $X_\Sigma$ is a {\it toric principal $G$-bundle} if $T$ acts on $\P$ lifting its action on $X_\Sigma$ in such a way that the $T$-action and the $G$-action on $\P$ commute (we will write the $T$-action on the left and the $G$-action on the right). By a {\it framed} toric principal $G$-bundle we mean a toric principal $G$-bundle $\P$ together with a choice of a point $p_0 \in \P_{x_0}$. The choice of $p_0 \in \P_{x_0}$ is equivalent to fixing a $G$-isomorphism between $\P_{x_0}$ and $G$ (where $G$ acts on itself via multiplication from right).
A {\it morphism of toric principal bundles} is a morphism of principal bundles that is also $T$-equivariant. A {\it morphism of framed principal bundles} is a morphism $F$ that sends the distinguished point $p_0 \in \P_{x_0}$ to the distinguished point $p'_0 \in \P'_{x_0}$. In other words, $F: \P_{x_0} \to \P'_{x_0}$ coincides with $\alpha: G \to G'$ after identifying $\P_{x_0}$, $\P'_{x_0}$ with $G$, $G'$ respectively. 


If $G$ is a subgroup of $\GL(E)$ then 
one defines the {\it associated vector bundle} $\E$ of $\P$ as the fiber product $\E = \P \times_G E$. It is straightforward to verify that it is indeed a toric vector bundle.

The next theorem is the main result of the paper. 

\begin{theorem}  \label{th-principal-G-bundle-plm}
Let $X_\Sigma$ be a toric variety and $G$ a reductive algebraic group over $\k$. There is a one-to-one correspondence between the isomorphism classes of framed toric principal $G$-bundles $\P$ over $X_\Sigma$ and the integral piecewise linear maps $\Phi:|\Sigma| \to \tilde{\B}(G)$. 

{Moreover, let $\alpha:G \to G'$ be a homomorphism of reductive algebraic groups. Let $\mathcal{P}$ (respectively $\mathcal{P}'$) be a framed toric principal $G$-bundle (respectively $G'$-bundle) with corresponding piecewise linear map $\Phi: |\Sigma| \to \tilde{\B}(G)$ (respectively $\Phi': |\Sigma| \to \tilde{\B}(G')$). Then there is a morphism of framed toric principal bundles $F: \mathcal{P} \to \mathcal{P}'$, that is equivariant with respect to $\alpha$, if and only if $\Phi' = \alpha_*(\Phi)$.}
	
When the base field $\k$ is $\C$, the theorem holds for any linear algebraic group $G$. 
\end{theorem}


In fact, we prove a more general version of Theorem 
\ref{th-principal-G-bundle-plm} that does not require $G$ to be reductive (Theorem \ref{th-principal-G-bundle-plm-v2} below). Most of the rest of this section is devoted to its proof. Theorem \ref{th-principal-G-bundle-plm} follows from Theorem \ref{th-principal-G-bundle-plm-v2}. We need the following definition.

\begin{definition}[Locally equivariantly trivial toric principal bundle] \label{def-local-equiv-trivial}
Let $G$ be a linear algebraic group. 
(1) Let $\P$ be a toric principal $G$-bundle on an affine toric variety $X_\sigma$. We say that $\P$ is \textit{equivariantly trivial} if there is a $T$-equivariant principal $G$-bundle isomorphism between $\P$ and a trivial toric principal bundle $X_\sigma \times G$ where $T$ acts on $X_\sigma \times G$ diagonally, acting on $X_\sigma$ in the usual way and on  $G$ via a homomorphism $T \to G$. That is, for any $t \in T$, $x \in X_\sigma$ and $g \in G$ we have: 
\begin{equation}  \label{equ-equiv-trivial}
t \cdot (x, g) = (t \cdot x, \phi_\sigma(t) g). 
\end{equation}
(2) Let $\P$ be a toric principal $G$-bundle on a toric variety $X_\Sigma$. We say that $\P$ is \textit{locally equivariantly trivial} if for each cone $\sigma \in \Sigma$, the restricted bundle $\P_{|X_\sigma}$ is equivariantly trivial.   
\end{definition}

\begin{theorem}  \label{th-principal-G-bundle-plm-v2}
Let $G$ be a linear algebraic group over $\k$. There is a one-to-one correspondence between the isomorphism classes of framed toric principal $G$-bundles $\P$ over $X_\Sigma$ that are locally equivariantly trivial and the integral piecewise linear maps $\Phi:|\Sigma| \to \tilde{\B}(G)$. 

{Moreover, let $\alpha:G \to G'$ be a homomorphism of linear algebraic groups. Let $\mathcal{P}$ (respectively $\mathcal{P}'$) be a locally equivariantly trivial framed toric principal $G$-bundle (respectively $G'$-bundle) with corresponding piecewise linear map $\Phi: |\Sigma| \to \tilde{\B}(G)$ (respectively $\Phi': |\Sigma| \to \tilde{\B}(G')$). Then there is a morphism of framed toric principal bundles $F: \mathcal{P} \to \mathcal{P}'$, that is equivariant with respect to $\alpha$, if and only if $\Phi' = \alpha_*(\Phi)$.}
\end{theorem}


{Theorem \ref{th-principal-G-bundle-plm} follows from Theorem \ref{th-principal-G-bundle-plm-v2} provided that we show the equivariant triviality of toric principal $G$-bundles over affine toric varieties holds for reductive $G$. This is proved in \cite[Theorem 4.1]{Biswas-Tannakian} which itself relies on \cite[Prop. 8.5]{BR}. We state this result below and give a sketch of its proof following \cite[Theorem 4.1]{Biswas-Tannakian}.}

{
\begin{theorem}  \label{th-equiv-trivial-affine}
Let $G$ be a reductive group. Let $\P$ be a toric principal $G$-bundle over an affine toric variety $X_\sigma$. Then $\P$ is equivariantly trivial. 
\end{theorem}
\begin{proof}[Sketch of proof]
Without loss of generality one can assume that $X_\sigma$ contains a (unique) torus fixed point $x_\sigma$. Then the $(T \times G)$-variety $\P$ contains a unique closed $(T \times G)$-orbit, namely the fiber $\P_{x_\sigma}$. Take $p_\sigma \in \P_{x_\sigma}$. One verifies that the first projection $T \times G \to T$ maps the stabilizer subgroup $(T \times G)_{p_\sigma}$ isomorphically to $T$, and hence there is a unique homomorphism $\phi: T \to G$ such that $(T \times G)_{p_\sigma} = \{(t, \phi(t)) \mid t \in T\}$.
In particular, the stabilizer $(T \times G)_{p_\sigma}$ is reductive. We can now apply the slice theorem in \cite[Prop. 8.5]{BR} to conclude that $\P$ is a fiber product $F \times_T (T \times G)$, for some affine $T$-variety $F$. Here $T$ acts on $T \times G$ by multiplication from left via $t \mapsto (t, \phi(t))$. One then sees that $F \cong \P / G \cong X_\sigma$ and $\P \cong X_\sigma \times G$ as required. 
\end{proof}
}	 
 
\begin{proof}[Proof of Theorem \ref{th-principal-G-bundle-plm-v2}]
For a cone $\sigma \in \Sigma$ let $T_\sigma$ be the stabilizer of the $T$-orbit $O_\sigma$. 
The subgroup $T_\sigma$ is generated by one-parameter subgroups corresponding to points in $\sigma \cap N$. Note that when $\sigma$ is full dimensional $T_\sigma=T$. We also let $x_\sigma \in O_\sigma$ be the  point in the closure of $T_\sigma \cdot x_0$. 

Let $\Phi: |\Sigma| \to \tilde{\B}(G)$ be an integral piecewise linear map. By assumption, for each cone $\sigma \in \Sigma$, the restriction $\Phi_{|\sigma}: \Lambda^\vee(T_\sigma) \to \Lambda^\vee(H_\sigma)$ is a $\Z$-linear map. Let $$\phi_\sigma: T_\sigma \to H_\sigma \subset G$$ be the homomorphism whose derivative at identity is $\Phi_{|\sigma}$. 

We prove the theorem in several steps.

{\bf Step 1:} Given an integral piecewise linear map $\Phi: |\Sigma| \to \tilde{\B}(G)$ we would like to construct a toric principal $G$-bundle $\P_\Phi$ over $X_\Sigma$. We construct $\P_\Phi$ by gluing equivariantly trivial $G$-bundles over affine toric charts. 

Take a cone $\sigma \in \Sigma$ with $X_\sigma$ its associated affine toric variety. Consider an arbitrary extension of the homomorphism $\phi_\sigma$ to the whole $T$, that is, $\phi_\sigma: T \to H_\sigma$. Let $$\P_\sigma = X_\sigma \times G$$ be the trivial $G$-bundle on $X_\sigma$ where $G$ acts on the second component by multiplication from the right.
Define an action of $T$ on $X_\sigma \times G$ by letting $T$ act on $X_\sigma \times G$ diagonally where it acts on $G$ by multiplication from left via the homomorphism $\phi_\sigma$ (see \eqref{equ-equiv-trivial}). 
We have defined toric principal bundles $P_\sigma$ on the affine charts $X_\sigma$, now we define gluing maps. Let $\sigma, \sigma'$ be two cones with $\tau = \sigma \cap \sigma'$. 
We would like to define a transition map $\psi = \psi_{\sigma, \sigma'}: X_\tau \to G$ so that the morphism $\tilde{\psi} = \tilde{\psi}_{\sigma, \sigma'}: {\P_\sigma}_{|X_\tau} \to {\P_{\sigma'}}_{|X_\tau}$ given by
$$\tilde{\psi}(x, g) = (x, \psi(x) g),$$ intertwines the $T$-actions on $\P_\sigma$ and $\P_{\sigma'}$. This is equivalent to the following: 
\begin{equation}   \label{equ-psi-x-t}
\psi(t \cdot x) \phi_\sigma(t) \psi(x)^{-1} = \phi_{\sigma'}(t),~ \forall t \in T,~\forall x \in X_\tau = X_\sigma \cap X_{\sigma'}.
\end{equation}
Recall that we fixed a point $x_0$ in the open orbit $X_0$. Then \eqref{equ-psi-x-t} is equivalent to: 
\begin{equation}   \label{equ-psi-x0}
\psi(t \cdot x_0) \phi_\sigma(t) \psi(x_0)^{-1} = \phi_{\sigma'}(t),~\forall t \in T.
\end{equation}
To construct such a transition map $\psi: X_\tau \to G$ we define $\psi$ on the open torus orbit $X_0$ by:
\begin{equation}  \label{equ-def-psi}
\psi(t \cdot x_0) = \phi_{\sigma'}(t) \phi_{\sigma}(t)^{-1},~\forall t \in T.
\end{equation}
Note that, in particular, we set $\psi(x_0)=1$. Letting $\psi(x_0)=1$ amounts to choosing a framing, that is, identifying the fiber at $x_0$ of $\P_{\Phi}$ with $G$. It is clear that the equality \eqref{equ-psi-x0} holds. 
Moreover, from the definition (Equation \eqref{equ-def-psi}) one sees that the $\tilde{\psi}_{\sigma, \sigma'}$ satisfy the cocycle condition. 

We only need to show that $\psi$ extends to a regular map on the whole $X_\tau$. This follows from the next lemma.
 \begin{lemma}  \label{lem-extension}
 With notation as above, let $\phi, \phi': T \to \GL(E)$ be linear representations of $T$ on a finite dimensional vector space $E$. The function $\psi(t \cdot x_0) = \phi'(t) \phi^{-1}(t)$ extends to a regular function $X_\tau \to \GL(E)$, 
 if and only if for every one-parameter subgroup $\lambda_a$, $a \in \tau \cap N$, the one-parameter subgroups $\phi \circ \lambda_a$ and $\phi' \circ \lambda_a$ are equivalent, i.e. the function $(\phi' \circ \lambda_a)(\phi \circ \lambda_a)^{-1}: \G_m \to \GL(E)$ extends to a regular function $\mathbb{A}^1 \to \GL(E)$. 
 \end{lemma}
\begin{proof}
	The ``only if'' direction is obvious, we prove the other direction.
	Let $B=\{b_1, \ldots, b_r\}$ (respectively $B'=\{b'_1, \ldots, b'_r\}$) be a basis of $T$-weight vectors for $\phi$ (respectively $\phi'$). For every $i$, let $u_i$ (respectively $u'_i$) be the weight of  $b_i$ (respectively $b'_i$). Let us write $b_i = \sum_j c_{ij} b'_j$. Then, for $a \in \tau \cap N$, we have:
	$$\phi'(\lambda_a(s))\phi(\lambda_a(s))^{-1}(b_i) = \sum_j c_{ij} s^{\langle u'_j - u_i, a \rangle} b'_j.$$
	This function extends to a regular function $\mathbb{A}^1 \to \End(E)$ implies that whenever $c_{ij} \neq 0$ we have $\langle u'_j - u_i, a \rangle \geq 0$. 
	On the other hand, we have:
	$$\phi'(t) \phi^{-1}(t)(b_i) = \sum_j c_{ij} \chi^{u'_j - u_i} b'_j.$$
	But a character $\chi^{u'_j - u_i}$ extends to a regular function on $X_\tau$ if and only if $\langle u'_j-u_i, a\rangle \geq 0$, for all $a \in \tau \cap N$. This shows that $\phi'\phi^{-1}$ extends to $X_\tau \to \End(E)$. Since  the same applies to $(\phi' \phi^{-1})^{-1} = \phi \phi'^{-1}$, we conclude that the image of extension of $\phi'\phi^{-1}$ to $X_\tau$ lies in $\GL(E)$. This proves the lemma.
\end{proof}
Fix a faithful representation $G \hookrightarrow \GL(E)$. Lemma \ref{lem-extension} shows that $\psi=\psi_{\sigma, \sigma'}: X_0 \to G \hookrightarrow \GL(E)$ extends to a regular function $\psi: X_\tau \to \GL(E)$. But $G$ is closed in $\GL(E)$ and thus the image of the extension $\psi$ lands in $G$. 
This finishes the construction of $\mathcal{P}_\Phi$.

Note that the above construction of $\P_\Phi$ depends on the choice of extensions $\phi_\sigma:T \to H_\sigma$. It remains to be checked that different such choices give $T$-equivariantly isomorphic principal bundles. For each $\sigma \in \Sigma$ let $\phi'_\sigma: T \to H_\sigma$ be another choice of an extension of $\phi_\sigma$ to the whole $T$. Let $\P'_\sigma$ and $\P'_\Sigma$ denote the resulting toric principal $G$-bundles on $X_\sigma$ and $X_\Sigma$ respectively. As above one verifies that the map:
$$(t \cdot x_0, g) \mapsto (t \cdot x_0, \phi_\sigma'(t) \phi_\sigma(t)^{-1} g)$$ extends to a $T$-equivariant isomorphism $\P_\sigma \to \P'_\sigma$ and these glue together to give a $T$-equivariant isomorphism $\P_\Phi \to \P'_\Phi$.
 
{\bf Step 2:} In the other direction, we would like to associate an integral piecewise linear map $\Phi_\P: |\Sigma| \to \tilde{\B}(G)$ to a framed toric principal $G$-bundle $\P$ on $X_\Sigma$ that is locally equivariantly trivial.
The bundle $\P$ is determined by its restrictions $\P_{|X_\sigma}$, $\sigma \in \Sigma$, and its transition functions $\psi_{\sigma, \sigma'}: X_\sigma \cap X_{\sigma'} \to G$, $\sigma, \sigma' \in \Sigma$. The transition functions satisfy the usual cocycle condition as well as the equivariance condition \eqref{equ-psi-x-t}. 
 
For each $\sigma \in \Sigma$, the local equivariant triviality states that the bundle $\P_\sigma := \P_{|X_\sigma}$ is $T$-equivariantly isomorphic to the trivial bundle $X_\sigma \times G$ where $T$ acts on $G$ via a homomorphism $\phi_\sigma: T \to G$. The homomorphisms $\phi_\sigma$ in turn yield $\Z$-linear maps $\Phi_\sigma: N \to \tilde{\B}_\Z(G)$. We wish to show that the local equivariant trivializations can be chosen so that the resulting $\Phi_\sigma$ glue together to produce a piecewise linear map $\Phi_\P: |\Sigma| \to \tilde{\B}(G)$.

Under the local equivariant trivialization $\P_\sigma \xrightarrow{\cong} X_\sigma \times G$, let the distinguished point (the frame) $p_0 \in \P_{x_0}$ map to $(x_0, g_\sigma)$. After replacing $\phi_\sigma$ with $g_\sigma^{-1} \phi_\sigma g_\sigma$ we can assume that $p_0$ maps to $(x_0, 1)$, where $1$ denotes the identity element in $G$.
{After doing this for all $\sigma \in \Sigma$, 
we see that $\psi_{\sigma, \sigma'}(x_0) = 1$, for all $\sigma, \sigma' \in \Sigma$.} 
The equivariance condition \eqref{equ-psi-x-t}  then implies that $\psi_{\sigma, \sigma'}$ on the open orbit $X_0$ is given by \eqref{equ-def-psi}. We also know that $\psi_{\sigma, \sigma'}$ is defined on the whole $X_\tau = X_\sigma \cap X_{\sigma'}$, where $\tau = \sigma \cap \sigma'$. 
Since the stabilizer subgroup $T_\tau$ is generated by the one-parameter subgroups $\lambda_a$, $a \in \tau \cap N$, we see that for any $a \in \tau \cap N$, the limit 
$$\lim_{s \to 0} \psi_{\sigma, \sigma'}(\lambda_a(s) \cdot x_0) = \lim_{s \to 0} \phi_{\sigma'}(\lambda_a(s))\phi_{\sigma}(\lambda_a(s))^{-1}$$ exists. This means that $\phi_{\sigma} \circ \lambda_a$ and $\phi_{\sigma'} \circ \lambda_a$ are equivalent one-parameter subgroups which then implies that the linear maps $\Phi_\sigma$ and $\Phi_{\sigma'}$ coincide on the cone $\tau = \sigma \cap \sigma'$. In conclusion, the $\Phi_\sigma$, $\sigma \in \Sigma$, glue together to produce a piecewise linear map $\Phi_\P: |\Sigma| \to \tilde{\B}(G)$.

Finally, we show that the above construction of $\Phi_\P$ is well-defined, that is, it does not depend on the choice of local equivariant trivializations. We note that the piecewise linearity implies that $\Phi_\P$ is uniquely determined by its values on the rays in $\Sigma$. Thus it suffices to show that these values are determined by the bundle $\P$. For a ray $\rho \in \sigma(1)$, let $\lambda_{\rho}: \mathbb{G}_m\rightarrow T$ be the one-parameter subgroup corresponding to the primitive vector ${\bf v}_{\rho}\in\rho$. We would like to show that for any cone $\sigma$ and any ray $\rho \in \sigma(1)$, the one-parameter subgroup $\phi_{\sigma} \circ \lambda_{\rho}: \mathbb{G}_m\rightarrow G$ is uniquely determined, up to equivalence of one-parameter subgroups, by the bundle $\mathcal{P}$. Fix a faithful representation $G \hookrightarrow \GL(E)$ where $E$ is a  finite dimensional vector space. Let $\mathcal{E} = \mathcal{P} \times_G E$ be the associated toric vector
bundle. Notice that $\mathcal{P}_{\sigma}\times_GE\cong X_{\sigma}\times G\times_GE\cong X_{\sigma}\times E$ with the action of $T$ on $E$ induced by $\phi_{\sigma}$. As in Klyachko's classification (\cite[\S 2]{Klyachko}) it follows that the $\Z$-filtration, and hence the flag and weights, of the one-parameter subgroup $\phi_\sigma \circ \lambda_\rho$ acting on $E$ are uniquely determined by the toric bundle $\mathcal{P}$ (cf. Remark \ref{rem-bldg-GL(r)-filtrations} ). This shows that the equivalence class of $\phi_\sigma \circ \lambda_\rho$ is determined by $\mathcal{P}$ as required.  



{\bf Step 3:} Next we verify that $\Phi \mapsto \P_\Phi$ and $\P \mapsto \Phi_\P$ are inverses of each other. 
Let $\Phi$ be an integral piecewise linear map with corresponding toric principal bundle $\P = \P_\Phi$. As above, for each cone $\sigma$ let $\phi_\sigma: T \to G$ be an extension of the homomorphism $T_\sigma \to G$ corresponding to the linear map $\Phi_\sigma = \Phi_{|\sigma}$. By construction in Step 1, $\P_{|X_\sigma} = X_\sigma \times G$ and the $T$-action on this trivializing chart is given by $\phi_\sigma$. From construction in Step 2, it is immediate that the piecewise linear map associated to $\P$ is $\Phi$. 

Conversely, let $\P$ be a toric principal bundle with the corresponding piecewise linear map $\Phi = \Phi_\P$. Let $\sigma, \sigma' \in \Sigma$ with $\tau = \sigma \cap \sigma'$ and let $\psi = \psi_{\sigma, \sigma'}: X_\tau \to G$ be the corresponding transition function. As in Step 2, for all $\sigma, \sigma'$, we can arrange for $\psi(x_0)$ to be equal to $1$ and thus $\psi(t \cdot x_0) = \phi_{\sigma'}(t) \phi_\sigma(t)^{-1}$ by \eqref{equ-psi-x0}. The construction in Step 1 then shows that $\P$ coincides with $\P_\Phi$.

To finish the proof of Theorem \ref{th-principal-G-bundle-plm}, it remains to prove the claim about the equivalence of categories. 

{\bf Step 4:} 
Let $G$, $G'$ be linear algebraic groups and let $\Phi$, $\Phi'$ be piecewise linear maps to $\tilde{\B}(G)$, $\tilde{\B}(G')$ respectively. {Let $\alpha: G \to G'$ be a homomorphism and let $\Phi' = \alpha_*(\Phi)$.} We construct a $T$-equivariant morphism of principal bundles $F: \P \to \P'$ as follows. For each $\sigma \in \Sigma$ define $F_\sigma: X_\sigma \times G \to X_\sigma \times G'$ by: 
\begin{equation}  \label{equ-alpha-F}
F_\sigma(x, g) = (x, \alpha(g)).
\end{equation}
One verifies that the $F_\sigma$ glue together to give a well-defined morphism $F: \P \to \P'$ and $F$ is $T$-equivariant. 

{\bf Step 5:} Conversely, let $\alpha: G \to G'$ be a homomorphism of linear algebraic groups. Let $\P$ and $\P'$ be toric principal $G$ and $G'$ bundles over $X_\Sigma$ respectively. {Let $F: \P \to \P'$ be a morphism of principal bundles that is equivariant with respect to $\alpha$.} If $\Phi$ is the piecewise linear map corresponding to $\P$, we wish to show that $\P'$ corresponds to the piecewise linear map $\hat{\alpha} \circ \Phi$. Let $\rho \in \Sigma(1)$ be a ray. On the open affine chart $X_\rho$, the bundle $\P$ (respectively $\P'$) can be identified with $X_\rho \times G$ (respectively $X_\rho \times G'$) and the action of $T$ is given by a homomorphism $\phi = \phi_\rho: T \to G$ (respectively $\phi' = \phi_{\rho'}: T \to G'$). Moreover, the distinguished point $p_0 \in \P_{x_0}$ (respectively $p'_0 \in \P'_{x_0}$) is identified with $(x_0, 1_G)$ (respectively $(x_0, 1_{G'})$). Here $1_G$ and $1_{G'}$ denote the identity elements of $G$ and $G'$ respectively.
Since $F$ is a bundle map, there is $f: X_\rho \times G \to G'$ such that $F: \P_{|X_\rho} \to \P'_{|X_\rho}$ is given by $(x, g) \mapsto (x, f(x, g))$. The assumption that $F$ sends the distinguished point $p_0$ to the distinguished point $p'_0$ means that $f(x_0, 1_G) = 1_{G'}$. The equivariance condition of $F: \P_{|X_\rho} \to \P'_{|X_\rho}$ then implies that, for any $t \in T$ and $g \in G$ we have:
\begin{equation}   \label{equ-f-equiv}
	f(t \cdot x_0, \phi(t)g) = \phi'(t) f(x_0, 1_G) \alpha(g)=\phi'(t) \alpha(g).
\end{equation}
Letting $g=\phi(t)^{-1}$ we conclude:
\begin{equation} \label{equ-f-equiv-X_0}
	f(t \cdot x_0, 1_G) = \phi'(t) \alpha(\phi(t))^{-1}.
\end{equation}
Since $f$ is regular on the whole $X_\rho \times G$ we see that the limit of $\phi'(t) \alpha(\phi(t))^{-1}$ exists in $G'$ as $t \cdot x_0 \to x_\rho$. This means that $\alpha \circ \phi \circ \lambda_\rho$ and $\phi' \circ \lambda_\rho$ are equivalent one-parameter subgroups of $G'$, where $\lambda_\rho$ is the one-parameter subgroup of $T$ associated to the primitive vector ${\bf v}_\rho$. It follows that $\Phi' = \hat{\alpha} \circ \Phi$. This finishes the proof.   

{\bf Step 6:} Finally, we need to show the maps $(\alpha, \Phi) \mapsto (\alpha, F)$ and $(\alpha, F) \mapsto (\alpha_*: \Phi \to \Phi')$ are inverses of each other. This follows from  \eqref{equ-f-equiv-X_0} and the fact that $\Phi' = \hat{\alpha} \circ \Phi$.
\end{proof}

\begin{remark}  \label{rem-morphisms-vb-vs-pb}
We point out that morphisms of vector bundles do not correspond to morphisms of principal bundles for general linear groups. 
This is because a linear map between vector spaces $E$ and $E'$ does not correspond to a group homomorphism between $\GL(E)$ and $\GL(E')$. Thus, our description of morphisms of (framed) toric principal bundles, in case of general linear groups, is different from Klyachko's description of morphisms of toric vector bundles. Morphisms between (framed) toric principal bundles for general linear groups are much more restricted than morphisms between toric vector bundles. When $G=G'=\GL(r)$ and $\alpha: \GL(r) \to \GL(r)$ is the identity, the morphisms of toric principal bundles on $X_\Sigma$ correspond to the isomorphisms of rank $r$ toric vector bundles on $X_\Sigma$.  
\end{remark}

In the recent paper \cite[Section 5]{BDDKP}, a description of morphisms for the category of (non-framed) toric principal $G$-bundles is given. It follows the Kaneyama type description of toric principal bundles in \cite{Biswas, Biswas-Serre} and is in terms of certain subsets of the group $G$ indexed by the cones in the fan and satisfying certain conditions. The morphisms we consider and those considered in \cite{BDDKP} differ in the following ways: firstly in Theorem \ref{th-principal-G-bundle-plm} we consider morphisms between principal bundle for possibly different groups $G$ and $G'$ {and equivariant with respect to} a homomorphism $\alpha: G \to G'$. On the other hand, \cite{BDDKP} fixes a group $G$ and consider the morphisms between toric $G$-principal bundles. That is, for them $G=G'$ and $\alpha$ is the identity. Secondly, we consider the morphisms of \textit{framed} toric principal bundles while \cite{BDDKP} do not fix a frame (at the distinguished point $x_0$). Hence the set of morphisms they consider is larger than ours. 
For example, the automorphism group of a framed toric principal bundle (with respect to the identity homomorphism) is trivial, while the automorphism group of a non-framed toric principal bundle is in general an intersection of certain parabolic subgroups (see \cite[Theorem 5.5]{BDDKP}). 
\begin{example}[Toric vector bundles]
\label{sec-toric-vb-PLM}
Rank $r$ vector bundles correspond to principal $\GL(r)$-bundles. Thus, Theorem \ref{th-principal-G-bundle-plm} applied to $G = \GL(r, \k)$ gives a classification of toric rank $r$ vector bundles. In fact, this exactly recovers Klyachko's classification. 
As in Section \ref{sec-prelim-toric-vb} let $\E$ be a toric vector bundle on a toric variety $X_\Sigma$. Fix a point $x_0$ in the open orbit $X_0 \subset X_\Sigma$ and let $E = \E_{x_0}$. In Klyachko's classification, for each ray $\rho \in \Sigma(1)$ we have a decreasing filtration $(E^\rho_i)_{i \in \Z}$ in $E$. Any such filtration defines a labeled flag in $E$ {(or equivalently an equivalence class of one-parameter subgroups of $\GL(E)$)}, see Remark \ref{rem-bldg-GL(r)-filtrations}. The compatibility condition (in Theorem \ref{th-Klyacjko}) translates to the condition that the labeled flags associated to the rays in the fan define a piecewise linear map $\Phi: |\Sigma| \to \tilde{\B}(\GL(E))$.


\end{example}

\begin{example}[Toric line bundles]   \label{ex-tlb}
In particular, let $\L$ be a toric line bundle on a toric variety $X_\Sigma$. Let $D = \sum_{\rho \in \Sigma(1)} a_\rho D_\rho$ be the torus invariant Cartier divisor corresponding to $\L$. The cone over the Tits building $\tilde{\B}(\GL(E)) = \tilde{\B}(\G_m)$ consists of one apartment which we identify with $\R$. Thus in this case, a piecewise linear map into $\tilde{\B}(\GL(E))$ is just a usual piecewise linear function. The piecewise linear function $\Phi$ corresponding to $\L$ is given by $\Phi(\v_\rho) = a_\rho$, $\forall \rho \in \Sigma(1)$. Recall that $\v_\rho \in N$ denotes the primitive vector along $\rho$.
We point out that several authors, for example \cite{Fulton, CLS} define the piecewise linear function associated to the divisor $D$ by $\Phi(\v_\rho) = -a_\rho$. This corresponds to taking increasing filtrations instead of decreasing filtrations in Klyachko's construction.
\end{example}

\begin{example}[Toric symplectic and orthogonal principal bundles]
\label{ex-orth-symp-bundle}
Let $\Sigma$ be a fan with the corresponding toric variety $X_\Sigma$.  
Let $G = \operatorname{O}(2r)$ or $\operatorname{Sp}(2r)$. In Example \ref{ex-building-orth-symp} we gave a description of $\tilde{\B}(G)$ in terms of isotropic flags. As an immediate corollary of Theorem \ref{th-principal-G-bundle-plm}, we obtain that the isomorphism classes of toric principal $G$-bundles on $X_\Sigma$ are in one-to-one correspondence with collections $\{ (F_{\rho, \bullet}, c_{\rho, \bullet}) \mid \rho \in \Sigma(1) \}$ of integral labeled isotropic flags (or equivalently filtrations by isotropic subspaces) that satisfy the following compatibility condition: for each cone $\sigma \in \Sigma$, there exists a normal frame 
$L_\sigma = \{L_{\sigma, 1}, \ldots, L_{\sigma, 2r} \}$ and a $\Z$-linear map $\Phi_\sigma: (\sigma \cap N) \to \Z^r$ such that for each ray $\rho \in \sigma(1)$ the labeled isotropic flag associated to $(L_\sigma, \Phi_\sigma(\v_\rho))$ coincides with $(F_{\rho, \bullet}, c_{\rho, \bullet})$. 

The case of $G = \operatorname{O}(2r+1)$ can be treated in a similar fashion.
\end{example}


\section{Characteristic classes}  \label{sec-char-class}
Finally, we describe the characteristic classes of a toric principal bundle in terms of its corresponding piecewise linear map. Extending the notion of Chern classes of a vector bundle, the characteristic classes of a principal bundle are given by the Chern-Weil homomorphism. Below we recall the equivariant Chern-Weil homomorphism and see how for a toric principal bundle it can be immediately recovered from the piecewise linear map associated to the bundle. 

First we consider the topological setting, namely when the base field is $\C$. Let $G$ be a complex  linear algebraic group with $T$ a maximal torus.
Let $BT$, $BG$ denote classifying spaces of $T$ and $G$ with $ET \to BT$ and $EG \to BG$ the corresponding universal bundles respectively. Note that $BT$ and $BG$ are topological spaces and are unique only up to homotopy. Recall that the bundle $EG \to BG$ has the universal property that for any $G$-bundle $\P \to X$ there exists a continuous map $f: X \to BG$ such that $f^*EG$ is isomorphic to $\P$.

Let $\P$ be a $T$-equivariant principal $G$-bundle on a $T$-variety $X$ (this means the action of $T$ on $\P$ lifts that of $X$ and commutes with the action of $G$ on $\P$). 
Consider $\P_T := \P \times_T ET$ and $X_T := X \times_T ET$. Recall that the equivariant cohomology $H^*_T(X)$ is the usual cohomology of $X_T$ (here cohomology is taken with coefficients in $\k = \C$). Then $\P_T$ is a principal $G$-bundle over $X_T$ and hence gives a map $f: X_T \to BG$. This then induces a homomorphism $f^*: H^*(BG) \to H^*(X_T) = H_T^*(X)$. When $G$ is reductive, the cohomology ring $H^*(BG)$ can be identified with $S(\mathfrak{g}^*)^G$, the $\C$-algebra of $\Ad_G$-invariant polynomials on the Lie algebra $\mathfrak{g} = \Lie(G)$. 
Alternatively, fix a maximal torus $H$ in $G$. Let $S(\mathfrak{h}^*)$ denote the $\C$-algebra of polynomials on the Lie algebra $\mathfrak{h}=\Lie(H)$. The inclusion $\mathfrak{h} \subset \mathfrak{g}$ induces an isomorphism of $S(\mathfrak{g}^*)^G$ with the algebra of Weyl group invariants $S(\mathfrak{h}^*)^W$ (this is sometimes known as the Chevalley restriction theorem). Hence $H^*(BG)$ can also be identified with $S(\mathfrak{h}^*)^W$. The equivariant Chern-Weil homomorphism is the homomorphism $$f^*: S(\mathfrak{g}^*)^G = S(\mathfrak{h}^*)^W \to H^*_T(X).$$

Next we consider the case where the base field is any algebraically closed field. In this case for the cohomology theory we take the Chow cohomology and consider characteristic classes in the Chow cohomology groups. 
Unfortunately there is no universal principal $G$-bundle in the category of algebraic varieties or schemes. This can be remedied by taking algebraic approximations to $EG \to BG$. Let $G$ be a linear algebraic group over $\k$. For any integer $m>0$ let $V_m$ be a finite dimensional $G$-module over $\k$ such that $G$ acts freely on a $G$-invariant open subset $U_m \subset V_m$ of codimension $> m$. Moreover, suppose that the geometric quotient $U_m/G$ exists. Then the groups $A^i(U_m/G)$, for degrees less than or equal to $m$, are independent (in a canonical way) of the representation $V_m$ and the open subset $U_m$ (see \cite[Theorem 1.1]{Totaro}).

In analogy with universal bundles we denote $U_m$ by $E_mG$ and its quotient $U_m/G$ by $B_mG$. Clearly, $E_mG \to B_mG$ is a principal $G$-bundle. It is an algebraic approximation of the universal principal bundle $EG \to BG$. Following \cite{Totaro, Edidin-Graham}, for $i \leq m$, the $i$-th $T$-equivariant Chow group of $X$ is defined to be $A^i_G(X) = A^i(X \times_G E_mG)$. The definition is well-defined i.e. independent of the choice the $G$-module $V_m$ and open set $U_m$.

As above fix a maximal torus $H \subset G$ and let $W$ be the Weyl group of $(G, H)$. Let $S_\R$ denote the $\R$-algebra generated by the character lattice of $H$. It is an important result that $A_G^*(\textup{pt})_\R$ is naturally isomorphic to the $\R$-algebra of $W$-invariants $S_\R^W$ (see \cite[Theorem 1(c)]{EG97}). 

We have the following universal property (see \cite[Lemma 1.6]{Totaro}):
\begin{lemma}
Let $\P \to X$ be a principal $G$-bundle. Then there is an affine space bundle $\pi:X' \to X$ and a map $f: X' \to B_mG$ such that the pullbacks $\pi^*\P$ and $f^*E_mG$ are isomorphic. 
\end{lemma}

We can use the above to define $T$-equivariant Chern-Weil homomorphism for Chow cohomology rings as follows. Let $P \to X$ be a $T$-equivariant principal $G$-bundle. As before, for $m>0$ sufficiently large, let $X_T:= X \times_T E_mT $ and $P_T:= \P \times_T E_mT$. Then $P_T \to X_T$ is a principal $T$-bundle and hence we can find $f: X' \to B_mG$ such that $f^*E_mG \cong \pi^*\P$.
Also we note that since $\pi: X' \to X$ is an affine space bundle, $\pi^*: A^*(X)_\R \to A^*(X')_\R$ is an isomorphism. The Chern-Weil homomorphism, for $i < m$, is given by:
$$S_\R^W \to A^i(X')_\R \cong A^i(X)_\R.$$ 


\begin{remark}
Alternatively, there is an algebraic universal principal $G$-bundle in the $2$-category of stacks. As $EG$ one takes the stack of a point $\textup{pt}$ equipped with a $G$-action and $BG$ the quotient stack $\textup{pt} / G$. One then defines the equivariant Chern-Weil homomorphism into Chow cohomology as in the topological setting explained above.
\end{remark}

We recall that the equivariant Chow cohomology ring $A^*_T(X_\Sigma)$ of a complete toric variety $X_\Sigma$ is naturally isomorphic to the algebra of piecewise polynomial functions on $N_\R$ with respect to the fan $\Sigma$ (see \cite{Payne-Chow-coh-toric}). The isomorphism is given by the localization map: 
$$A^*_T(X_\Sigma) \to \bigoplus_{\sigma \in \Sigma(n)} A^*_T(\{x_\sigma\})) \cong \bigoplus_{\sigma \in \Sigma(n)} S(T)_\R.$$ Here the map $A^*_T(X_\Sigma) \to A^*_T(\{x_\sigma\})$ is induced by the inclusion $\{x_\sigma\} \hookrightarrow X_\Sigma$, and $S(T)_\R$ denotes the $\R$-algebra generated by the character lattice of $T$.  

A $W$-invariant element $p \in S^W_\R$ defines a polynomial function on the cocharacter lattice of the maximal torus $H$. We note that every one-parameter subgroup of $G$ is conjugate to a one-parameter subgroup of the maximal torus $H$. By requiring that $p$ is conjugation invariant, we can extend it to a well-defined function on the whole $\tilde{\B}(G)$. That is, for one-parameter subgroup $\lambda: \G_m \to G$ we define $p(\lambda) = p(g \lambda g^{-1})$ where $g \in G$ is such that $g \lambda g^{-1}: \G_m \to H$. That this is well-defined is a consequence of the following standard lemma. We include a proof which is taken from \cite{mathoverflow}. 
\begin{lemma}
Let $\lambda$, $\lambda' \in \Lambda^\vee(H)$ be two one-parameter subgroups of a maximal torus $H$. Suppose there exists $g \in G$ such that $g \lambda g^{-1} = \lambda'$. Then there is $w \in W$ such that $w \lambda w^{-1} = \lambda'$.
\end{lemma}
\begin{proof}
Under the assumptions of the lemma, $g H g^{-1}$ centralizes the image of $\lambda'$. Thus $H$ and $g H g^{-1}$ are maximal tori in the centralizer of the image of $\lambda'$ and hence there is $h$ in this centralizer such that $g H g^{-1} = h H h^{-1}$. Now take $w$ to be the Weyl group element represented by $h^{-1}g \in N_G(H)$. 
\end{proof}

\begin{theorem}[Characteristic classes]   \label{th-Chern-Weil-toric-principal-bundle}
Let $\P$ be a toric principal $G$-bundle on a complete toric variety $X_\Sigma$ over $\k$ with the corresponding piecewise linear map $\Phi: |\Sigma| \to \tilde{\B}(G)$. Let $p \in S^W_\R$ be a $W$-invariant polynomial. Then the image of $p$ under the equivariant Chern-Weil homomorphism is given by the piecewise polynomial function $p \circ \Phi$.
\end{theorem}
\begin{proof}
We need to show that for any maximal cone $\sigma \in \Sigma(n)$ the restriction of image of $p$ to the fixed point $x_\sigma$ coincides with $(p \circ \Phi)_{|\sigma}$. The fixed point $x_\sigma$ gives a section $B_mT \to X_{m,T}:=X \times_T E_mT$ which induces a homomorphism $A^i_T(X)_\R \to A^i(B_mT)_\R \cong S(T)_{\R, i}$, for $i < m$. We thus obtain a homomorphism: $$S^W_{\R, i} \cong A^i(B_mG)_\R \to A^i_T(X)_\R \to A^i(B_mT)_\R \cong S(T)_{\R, i}.$$ Recall that the action of $T$ on $\P_{|X_\sigma}$ is given by a homomorphism $\phi_\sigma: T \to \P_{x_0} \cong G$. 
One verifies that this induces the above homomorphism $S^W_{\R, i} \to S(T)_{\R, i}$. This finishes the proof.
\end{proof}

Let us consider the case of $G=\GL(r)$ and toric vector bundles. One can naturally define elementary symmetric functions on $\tilde{\B}(\GL(E))$ as we now describe. For $1 \leq i \leq r$ {let $\epsilon_i: \R^r \to \R$ denote the $i$-th elementary symmetric polynomial}, i.e.
$$\epsilon_i(x_1, \ldots, x_r) = \sum_{1\leq j_1 < \cdots < j_i \leq r} x_{j_1} \cdots x_{j_i}.$$
Each $\epsilon_i$ induces a well-defined real-valued function on $\tilde{\B}(\GL(r))$ which we again denote by $\epsilon_i$. 
Let us regard an element of $\tilde{\B}(\GL(r))$ as a labeled flag $(F_\bullet, c_\bullet)$, where $F_\bullet = (\{0\} = F_0 \subsetneqq F_1 \subsetneqq \cdots \subsetneqq F_k = \k^r)$ and $c_\bullet = (c_1 > \cdots > c_k)$ (see Example \ref{ex-building-GL_r}). Then the value of $\epsilon_i$ on this element is equal to the $i$-th elementary symmetric function on $c_1, \ldots, c_k$ where each $c_i$ is repeated $\dim(F_i / F_{i-1})$ times.
As a special case of Theorem \ref{th-Chern-Weil-toric-principal-bundle} we obtain the following description of equivariant Chern classes of toric vector bundles. This corollary is not new and appears, in a slightly different language, in \cite[Proposition 3.1]{Payne-moduli}.

\begin{corollary}[Equivariant Chern classes]    \label{cor-Chern-class-tvb} 
Let $\E$ be a toric vector bundle over a toric variety $X_\Sigma$ with $\Phi_\E: |\Sigma| \to \tilde{\B}(\GL(E))$ its corresponding piecewise linear map. Then for any $1 \leq i \leq r$, the $i$-th equivariant Chern class $c^T_i(\E)$ is represented by the piecewise polynomial function $\epsilon_i \circ \Phi_\E$.
\end{corollary}

\end{document}